\documentclass[11pt]{article}
\usepackage{epigamath}

\usepackage[notext]{kpfonts}
\usepackage{baskervald}

\setpapertype{A4}

\usepackage[english]{babel}

\usepackage[dvips]{graphicx}     

\title{\vspace{-0.5cm}Abundance for varieties with many differential forms}
\titlemark{Abundance for varieties with many differential forms}
\author{\vspace{0cm}Vladimir Lazi\'c and Thomas Peternell}
\authoraddresses{
\authordata{Vladimir Lazi\'c}{\firstname{Vladimir} \lastname{Lazi\'c}\\
\institution{Fachrichtung Mathematik, Campus, Geb\"aude E2.4, Universit\"at des Saarlandes, 66123 Saarbr\"ucken, Germany}\\
\email{lazic@math.uni-sb.de}}\\
\authordata{Thomas Peternell}{\firstname{Thomas}
\lastname{Peternell}\\
\institution{Mathematisches Institut, Universit\"at Bayreuth, 95440 Bayreuth, Germany}\\
\email{thomas.peternell@uni-bayreuth.de}}
}
\authormark{V. Lazi\'c and T. Peternell}
\date{\vspace{-5ex}} 
\journal{\'Epijournal de G\'eom\'etrie Alg\'ebrique} 
\acceptation{Received by the Editors on August 20, 2017, and in final form
on November 16, 2017. \\ Accepted on January 12, 2018.}

\newcommand{\articlereference}{Volume 2 (2018), Article Nr. 1}

 \usepackage[all]{xy}

\newtheorem{thm}{Theorem}[section]

\newtheorem{pro}[thm]{Proposition}

\newtheorem{lem}[thm]{Lemma}

\newtheorem{cor}[thm]{Corollary}

\newtheorem{thmA}{Theorem}

\newtheorem{dfn}[thm]{Definition}

\newtheorem{rem}[thm]{Remark}

\newcommand{\Z}{\mathbb{Z}}
\newcommand{\N}{\mathbb{N}}
\newcommand{\C}{\mathbb{C}}
\newcommand{\R}{\mathbb{R}}
\newcommand{\Q}{\mathbb{Q}}
\newcommand{\PS}{\mathbb{P}}

\newcommand{\OO}{\mathcal{O}}

\newcommand{\lto}{\longrightarrow}
\newcommand{\reg}{\mathrm{reg}}
\newcommand{\can}{\mathrm{can}}

\DeclareMathOperator{\codim}{codim}

\DeclareMathOperator{\mult}{mult}

\DeclareMathOperator{\Supp}{Supp}

\DeclareMathOperator{\Div}{Div}

\begin{document}


\maketitle



\begin{prelims}


\def\abstractname{Abstract}
\abstract{We prove that the abundance conjecture holds on a variety $X$ with mild singularities if $X$ has many reflexive differential forms with coefficients in pluricanonical bundles, assuming the Minimal Model Program in lower dimensions. This implies, for instance, that under this condition, hermitian semipositive canonical divisors are almost always semiample, and that klt pairs whose underlying variety is uniruled  have good models in many circumstances. When the numerical dimension of $K_X$ is $1$, our results hold unconditionally in every dimension. We also treat a related problem on the semiampleness of nef line bundles on Calabi-Yau varieties.}




\languagesection{Fran\c{c}ais}{%

\textbf{Titre. Abondance pour des vari\'et\'es avec beaucoup de formes diff\'erentielles} \commentskip \textbf{R\'esum\'e.} Nous \'etablissons la conjecture d'abondance pour une vari\'et\'e
$X$ avec des singularit\'es mod\'er\'ees ayant beaucoup de formes diff\'erentielles \`a coefficients dans des puissances du fibr\'e canonique, en admettant la validit\'e du Programme des Mod\`eles Minimaux en dimensions strictement inf\'erieures. Cela implique par exemple, que, sous cette condition, les diviseurs canoniques hermitiens semi-positifs sont presque toujours semi-amples, et que les paires klt dont la vari\'et\'e sous-jacente est unir\'egl\'ee ont de bons mod\`eles dans de nombreuses situations. Lorsque la dimension num\'erique de $K_X$ est $1$, nos r\'esultats sont vrais inconditionnellement en toute dimension. Nous traitons \'egalement un probl\`eme similaire concernant la semi-amplitude des fibr\'es en droites nef sur les vari\'et\'es de Calabi-Yau.}

\end{prelims}


\newpage

\setcounter{tocdepth}{1}
\tableofcontents

\section{Introduction}

In this paper we prove that the abundance conjecture holds on a variety $X$ with mild singularities if $X$ has many reflexive differential forms with coefficients in pluricanonical bundles, assuming the Minimal Model Program in lower dimensions. The abundance conjecture and the existence of good models are the main open problems in the Minimal Model Program in complex algebraic geometry. A main step towards abundance is the so called nonvanishing conjecture. Various theorems presented in this paper are the first results on nonvanishing in dimensions $\geq4$ when the numerical dimension of $X$ is not $0$ or $\dim X$. As a by-product, we obtain a new proof of (the most difficult part of) nonvanishing in dimension $3$.

Recall that given a $\Q$-factorial projective variety $X$ with terminal singularities (terminal variety for short),  the Minimal Model Program (MMP) predicts that either $X$ is uniruled and has a birational model which admits a Mori fibration, or $X$ has a birational model $X'$ with terminal singularities such that $K_{X'}$ is nef; the variety $X'$ is called a \emph{minimal model} of $X$. The \emph{abundance conjecture} then says that $K_{X'}$ is semiample, i.e.\ some multiple $mK_{X'}$ is basepoint free; the variety $X'$ is then a \emph{good model} of $X$. For various important reasons it is necessary to study a more general situation of klt pairs $(X,\Delta).$ 

\medskip

We recall briefly the previous progress towards the resolution of these conjectures, concentrating mainly on the case of klt singularities. Everything is classically known for surfaces. For terminal threefolds, minimal models were constructed in \cite{Mor88,Sho85} and abundance was proved in \cite{Miy87,Miy88a,Miy88b,Kaw92,Kol92}. The corresponding generalisations to threefold klt pairs were established in \cite{Sho92} and in \cite{KMM94}. Minimal models for canonical fourfolds exist by \cite{BCHM,Fuj05}, and abundance for canonical fourfolds is known when $\kappa(X,K_X)>0$ by \cite{Kaw85}. In arbitrary dimension, the only unconditional results are the existence of minimal models for klt pairs of log general type proved in \cite{HM10,BCHM} and in \cite{CL12a,CL13} by different methods, the abundance for klt pairs of log general type \cite{Sho85,Kaw85b}, and the abundance for varieties with numerical dimension $0$, see \cite{Nak04}. 

The non-log-general type case is notoriously difficult, and there are only several reduction steps known. The running assumption is that the Minimal Model Program holds in lower dimensions; this is of course completely natural, since the completion of the programme should be done by induction on the dimension. With this assumption in mind, minimal models for klt pairs $(X,\Delta)$ with $\kappa(X,K_X+\Delta)\geq0$ exist by \cite{Bir11}, and good models exist for klt pairs $(X,\Delta)$ with $\kappa(X,K_X+\Delta)\geq1$ by \cite{Lai11}. The existence of good models for arbitrary klt pairs was reduced to the existence of good models for klt pairs $(X,\Delta)$ with $X$ terminal and $K_X$ pseudoeffective in \cite{DL15}.

\medskip

By a result of Lai, which we recall below in Theorem \ref{thm:lai}, the existence of good models for klt pairs is reduced to proving the existence of good models of pairs $(X,\Delta)$ with $\kappa(X,K_X+\Delta)\leq0$. We exploit this fact crucially in this paper. There are two faces of the problem of existence of good models. 
\begin{itemize}
\item The first is \emph{nonvanishing}: showing that if $(X,\Delta)$ is a klt pair with $K_X+\Delta$ pseudoeffective, then $\kappa(X,K_X+\Delta)\geq0$. 
\item The second is \emph{semiampleness}: showing that if $(X,\Delta)$ is a klt pair with $\kappa(X,K_X+\Delta)\geq0$, then any minimal model of $(X,\Delta)$ is good. 
\end{itemize} 
By \cite{DHP13}, nonvanishing is reduced to proving $\kappa(X,K_X)\geq0$ for a smooth variety $X$ with pseudoeffective canonical class.

The approach to semiampleness until now has been to find a suitable extension result for pluricanonical forms as in \cite{DHP13,GM17}, which usually requires working with singularities on reducible spaces. On the other hand, the nonvanishing in dimensions greater than $3$ has remained completely mysterious. The proof by Miyaoka and Kawamata in dimension $3$ resists straightforward generalisation to higher dimensions because of explicit use of surface and $3$-fold geometry.

\medskip

In this work we take a very different approach. The main idea is that the growth of global sections of the sheaves $\Omega_X^{[q]}\otimes\OO_X(mK_X)$ should correspond to the growth of sections of $\OO_X(mK_X)$, where $\Omega_X^{[q]}=(\bigwedge^q \Omega^1_X)^{**}$ is the sheaf of reflexive $q$-differentials on a terminal variety $X$. We use recent advances on the structure of sheaves of $q$-differentials together with the Minimal Model Program for a carefully chosen class of pairs to establish this link; the details of the main ideas of the proof are at the end of this section. 

A hope that such a link should exist was present already in \cite[\S2.7]{DPS01}. The situation here is somewhat similar to the semiampleness conjecture for nef line bundles $\mathcal L$ on varieties $X$ of Calabi-Yau type,\label{page} and we developed our main ideas while thinking about this related problem, see \cite{LOP16}. In that context, it seems that the first consideration of the sheaves $\Omega_X^q\otimes\mathcal L$ appears in \cite{Wi94}, and parts of our proofs here are inspired by some arguments there. Sheaves of this form also appear in \cite{Ver10} in an approach towards the hyperk\"ahler SYZ conjecture.

\medskip

We now discuss our results towards nonvanishing and semiampleness. Below, $\nu(X,D)$ denotes the numerical dimension of a $\Q$-Cartier $\Q$-divisor on a normal projective variety $X$, see Definition \ref{dfn:kappa}. 

\subsection*{Nonvanishing}

The following is our first main result.

\begin{thmA}\label{thm:A}
Let $X$ be a terminal projective variety of dimension $n$ with $K_X$ pseudoeffective. Assume either 
\begin{enumerate}
\item[\rm (i)] the existence of good models for klt pairs in dimensions at most $n-1$, or 
\item[\rm (ii)] that $K_X$ is nef and $\nu(X,K_X)=1$. 
\end{enumerate}
Assume that there exists a positive integer $q$ such that
$$ h^0\big(X,\Omega^{[q]}_X \otimes \OO_X(mK_X)\big)>0$$
for infinitely many $m$ such that $mK_X$ is Cartier. Then $\kappa(X,K_X)\geq0$.
\end{thmA}

Recall that  the sheaf $\Omega_X^{[q]}$ of reflexive differentials is defined as the extension $i_*\Omega^q_{X_\reg}$ from the smooth locus $i\colon X_\reg\to X$.  Alternatively, by the results of \cite{GKKP11}, we have $\Omega_X^{[q]} = \pi_*\Omega^q_Y$ for any desingularization $\pi\colon Y \to X$. Hence, the assumption in Theorem \ref{thm:A} can be rephrased by saying that 
$$ h^0\big(Y,\Omega^q_Y \otimes \OO_Y(m\pi^*K_X)\big)>0$$
for some (or any) desingularisation $\pi\colon Y \to X$ and infinitely many $m$ such that $mK_X$ is Cartier.

 We actually prove a stronger statement involving any effective tensor representation of $\Omega_Y^q$, see Theorems \ref{thm:nonvanishingForms} and \ref{thm:nonvanishingFormsnu1} below. As mentioned above, the assumptions on the MMP in lower dimensions are natural and expected for any result towards abundance. 

Using the main results of \cite{DPS01} and \cite{GM17}, our application of Theorem \ref{thm:A} is the following, which shows in particular, that hermitian semipositive canonical divisors are almost always semiample, assuming the MMP in lower dimensions. It follows from Corollaries \ref{cor:nv}, \ref{cor:semipositive} and \ref{thm:nu1}.

\begin{thmA}\label{thm:B}
Let $X$ be a terminal projective variety of dimension $n$ with $K_X$ pseudoeffective and $\chi(X,\OO_X)\neq0$.
\begin{enumerate}
\item[\rm (i)] Assume the existence of good models for klt pairs in dimensions at most $n-1$. If $K_X$ has a singular metric with algebraic singularities and semipositive curvature current, then $\kappa(X,K_X) \geq0$. Moreover, if $K_X$ is hermitian semipositive, then $K_X$ is semiample. 
\item[\rm (ii)] Assume that $K_X$ is nef and $\nu(X,K_X)=1$. Then $\kappa(X,K_X) \geq0$. 
\end{enumerate} 
\end{thmA}

We comment briefly on the assumption $\chi(X,\OO_X) \ne 0$. In the proof of Theorem \ref{thm:B} we argue by contradiction and produce a desingularisation $\pi\colon Y \to X$ and a divisor $D$ on $Y$ such that $\chi(Y, \OO_Y(mD)) = 0$ for infinitely many $m$. In particular, $\chi(Y,\OO_Y) = \chi(X,\OO_X) = 0$, which gives a contradiction. If $X$ is smooth and $K_X$ is hermitian semipositive, then in fact $D = K_X$. In general, $D$ invokes the canonical bundles of $Y$ and $X$ and divisors coming from multiplier ideal sheaves. 

\subsection*{Semiampleness}

Let $(X,\Delta)$ be a klt pair such that $K_X+\Delta$ is pseudoeffective. By passing to a terminalisation, in order to prove the existence of a good model for $(X,\Delta)$ we may assume that the pair is terminal, and we distinguish two cases. If $K_X$ is not pseudoeffective, then by \cite{BDPP} the variety $X$ is uniruled, and by modifying $X$ via a generically finite map, by \cite{DL15} we may assume that $\kappa(X,K_X)\geq0$. If $K_X$ is pseudoeffective, then assuming nonvanishing, we may also assume that $\kappa(X,K_X)\geq0$.

Therefore, when considering the semiampleness problem, then -- assuming nonvanishing -- we may assume that the pair $(X,\Delta)$ under consideration is terminal and that $\kappa(X,K_X)\geq0$. In this context, the following is our second main result.

\begin{thmA}\label{thm:C}
Let $(X,\Delta)$ be a projective  terminal pair of dimension $n$. Assume either 
\begin{enumerate}
\item[\rm (i)] the existence of good models for klt pairs in dimensions at most $n-1$, and that $\kappa(X,K_X)\geq0$, or 
\item[\rm (ii)] that $\Delta=0$, $K_X$ is nef and $\nu(X,K_X)=1$. 
\end{enumerate}
Assume that there exist a resolution $\pi\colon Y\to X$ of $X$ and a positive integer $q$ such that
$$ h^0\big(Y,\Omega_Y^q(\log\lceil\pi^{-1}_*\Delta\rceil)\otimes \OO_Y(m\pi^*(K_X+\Delta))\big)>\binom{n}{q} \quad\text{for some }m.$$
Then $(X,\Delta)$ has a good model.
\end{thmA}

As above, our results are stronger and apply to any effective tensor representation of the sheaf of logarithmic differentials, see Theorems \ref{thm:abundanceForms} and \ref{thm:nonvanishingFormsnu2} below. Theorem \ref{thm:C} generalises \cite[Theorem 1.5]{Taj14}.

The following applications of Theorem \ref{thm:C} prove semiampleness for a large class of varieties. Theorem \ref{thm:D} follows from Corollaries \ref{cor:chi} and \ref{cor:xx}.

\begin{thmA}\label{thm:D}
Let $(X,\Delta)$ be a  terminal pair of dimension $n$ such that $|\chi(X,\OO_X)|>2^{n-1}$. Assume either 
\begin{enumerate}
\item[\rm (i)] the existence of good models for klt pairs in dimensions at most $n-1$, and that $\kappa(X,K_X)\geq0$, or 
\item[\rm (ii)] that $\Delta=0$, $K_X$ is nef and $\nu(X,K_X)=1$.
\end{enumerate}
Then $(X,\Delta)$ has a good model. 
\end{thmA}

When the underlying variety of a klt pair is uniruled, by using the main result of \cite{DL15} we can say much more: that it in many circumstances has a good model, assuming the MMP in lower dimensions.

\begin{thmA}\label{thm:E}
Assume the existence of good models for klt pairs in dimensions at most $n-1$. Let $(X,\Delta)$ be a klt pair of dimension $n$ such that $X$ is uniruled and $K_X+\Delta$ is pseudoeffective. If $\vert\chi(X,\OO_X)\vert> 2^{n-1}$, then $(X,\Delta)$ has a good model.
\end{thmA}

Theorem \ref{thm:E} follows from Theorem \ref{thm:uniruled} below.

\subsection*{Nef bundles on Calabi-Yau varieties}

As mentioned on page \pageref{page}, the problem of deciding whether the canonical class on a minimal variety is semiample is intimately related to the problem of deciding whether a nef line bundle on a variety with trivial canonical class is semiample. Using very similar ideas, one can find analogues of Theorems \ref{thm:A} and \ref{thm:C} in this second context; this is done in Section \ref{sec:CY}. This method was already crucial in \cite{LOP16}, and we expect it to be of similar use in the future. Immediate consequences of this approach are contained in the following result, which is Theorem \ref{cor:nef}.

\begin{thmA}\label{thm:G}
Assume the existence of good models for klt pairs in dimensions at most $n-1$. Let $X$ be a projective klt variety of dimension $n$ such that $K_X\sim_\Q0$, and let $\mathcal L$ be a nef line bundle on $X$. 
\begin{enumerate}
\item[\rm (i)] Assume that $\mathcal L$ has a singular hermitian metric with semipositive curvature current and with algebraic singularities. If $\chi(X,\OO_X)\neq0$, then $\kappa(X,\mathcal L)\geq0$.
\item[\rm (ii)] If $\mathcal L$ is hermitian semipositive and if $\chi(X,\OO_X)\neq0$, then $\mathcal L$ is semiample.
\end{enumerate} 
\end{thmA}

Note that when $X$ is a hyperk\"ahler manifold and $\mathcal L$ is a hermitian semipositive line bundle on $X$, then automatically $\chi(X,\OO_X)\neq0$, and it is known unconditionally that $\kappa(X,\mathcal L)\geq0$ by \cite{Ver10}. We obtain in Section \ref{sec:CY} also more precise information when $X$ is a Calabi-Yau manifold of even dimension.

\subsection*{Sketch of the proof}

We now outline the main steps of the proof of Theorem \ref{thm:A}; the proof of Theorem \ref{thm:C} is similar. Assume for simplicity that $X$ is a smooth minimal variety, and that  there exists a positive integer $q$ such that for infinitely many $m$ we have
$$ h^0(X,\Omega^q_X \otimes \OO_X(mK_X))>0.$$
By using the main results of \cite{CP11,CP15}, we first show that there exists a pseudoeffective divisor $F$ and divisors $N_m\geq0$ for infinitely many $m$ with
$$N_m\sim mK_X-F.$$
By the basepoint free theorem and by the pseudoeffectivity of $F$ we may assume that none of $N_m$ are big, and a short additional argument allows to deal with the case where $\kappa(X,N_m)=0$ for infinitely many $m$. Then the results of \cite{Lai11,Kaw91} allow us to run a Minimal Model Program with scaling for a carefully chosen pair $(X,\varepsilon N_k)$, so that $K_X$ is nef at every step of the programme. In other words, replacing $X$ by the obtained minimal model, we may additionally assume that all $K_X+\varepsilon N_m$ are semiample for $m$ sufficiently large. By using the pseudoeffectivity of $F$ again, we show the existence of a fibration $f\colon X\to Y$ to a lower-dimensional variety $Y$ and of a big $\Q$-divisor $D$ on $Y$ such that $K_X\sim_\Q f^*D$. Hence, $\kappa (X,K_X) \geq 0$. 

\subsection*{Outline of the paper}

In Section \ref{sec:prelim} we collect definitions and results which are used later in the paper. Most of the material should be well known, however we decided to give proofs when a good reference could not be found. 

The whole of Section \ref{sec:MMPtwist} is dedicated to the MMP argument hinted at above. The subtlety of the proof consists in finding the right Minimal Model Program to run, in order to preserve all the good properties of the canonical class and to find a birational model of the initial variety, on which the canonical class is the pullback of a $\Q$-divisor from a lower-dimensional base.

In Section \ref{sec:thmA} we prove a more general version of Theorem \ref{thm:A}(i). Apart from the results from Section \ref{sec:MMPtwist}, the main input is the stability of the cotangent bundle \cite{CP11,CP15}, which generalise previous results of \cite{Miy87a}. Using the main result of \cite{DPS01}, we also derive a part of Theorem \ref{thm:B}.

Section \ref{sec:thmC} is devoted to the proofs of Theorem \ref{thm:C} and a part of Theorem \ref{thm:D}. The proofs are similar to those in Section \ref{sec:thmA}, although they are somewhat more involved. 

Under the additional assumption that the numerical dimension of the (log) canonical class is $1$, we can deduce several results unconditionally, and this occupies Section \ref{sec:nd1}. In this special case, one can avoid the MMP techniques from Section \ref{sec:MMPtwist}. Previous unconditional results were only known when the numerical dimension is $0$ or maximal, see \cite{BCHM,Nak04,Dru11}. 

In Section \ref{sec:uniruled}, we consider uniruled varieties and prove Theorem \ref{thm:E}. Following \cite{DL15}, we reduce the existence of good models for uniruled pairs to the case of effective canonical class, so that we can apply the results from Section \ref{sec:thmC}.

Finally, we treat the related problem of the semiampleness of nef line bundles on varieties of Calabi-Yau type in Section \ref{sec:CY}. The techniques are similar to those of the previous sections. We finish by proving Theorem \ref{thm:G}.

\section{Preliminaries}\label{sec:prelim}

In this paper we work over $\C$, and all varieties are normal and projective. 
We often use without explicit mention that if $f\colon X\dashrightarrow Y$ is a birational map between $\Q$-factorial varieties which is either a morphism or an isomorphism in codimension $1$, and if $D$ is a big, respectively pseudoeffective divisor on $X$, then $f_*D$ is a big, respectively pseudoeffective divisor on $Y$.

We start with the following easy result which will be used in the proof of Theorem \ref{thm:MMPtwist}.

\begin{lem}\label{relation}
Let $X$ be a variety and $D$ a reduced Weil divisor on $X$. Let $\{N_m\}_{m\in \N}$ a sequence of effective Weil divisors on $X$ such that $\Supp N_m\subseteq\Supp D$ for every $m$. Then there exist positive integers $k\neq \ell$ such that $N_k\geq N_\ell$.
\end{lem}

\begin{proof}
Setting $D=\sum_{i=1}^n D_i$, where $D_i$ are prime divisors, the proof is by induction on $n$. We may assume that there does not exist an integer $k\geq2$ such that $N_k\geq N_1$. Then for each $k\geq2$, we have $\mult_{D_i}N_k<\mult_{D_i}N_1$ for some $i$. Hence, by passing to a subsequence and by relabelling, we may assume that $\mult_{D_1}N_k$ is constant for all $k\geq2$, and set 
$$N_k'=N_k-(\mult_{D_1}N_k)D_1.$$
Then $\Supp N_k'\subseteq\Supp(D-D_1)$, and the conclusion follows.
\hfill $\Box$
\end{proof}

\medskip

We need the following easy consequences of the Hodge index theorem and of the existence of the Iitaka fibration.

\begin{lem}\label{lem:hodge}
Let $X$ be a smooth projective surface, and let $L$ and $M$ be divisors on $X$ such that 
$$L^2=M^2=L\cdot M=0.$$ 
If $L$ and $M$ are not numerically trivial, then $L$ and $M$ are numerically proportional.
\end{lem} 

\begin{proof}
Let $H$ be an ample divisor on $X$. By the Hodge index theorem we have $\lambda=L\cdot H\neq0$ and $\mu=M\cdot H\neq 0$, and set $D=\lambda M-\mu L$. Then $D^2=D\cdot H=0$, hence $D\equiv 0$ again by the Hodge index theorem.
\hfill $\Box$
\end{proof}

\begin{lem}\label{lem:iitaka}
Let $X$ be a normal projective variety and let $L$ be a $\Q$-Cartier $\Q$-divisor on $X$ with $\kappa(X,L)\geq0$. Then there exist a resolution $\pi\colon Y\to X$ and a fibration $f\colon Y\to Z$:
\[
\xymatrix{ 
Y \ar[d]^{\pi} \ar[r]^{f} & Z \\
X  & 
}
\]
such that $\dim Z=\kappa(X,L)$, and for every $\pi$-exceptional $\Q$-divisor $E\geq0$ on $Y$ and for a very general fibre $F$ of $f$ we have
$$\kappa\big(F,(\pi^*L+E)|_F\big)=0.$$
\end{lem}

\begin{proof}
By passing to multiples, we may assume that $L$ and $E$ are Cartier and not only $\Q$-Cartier.  The result is clear when $\kappa(X,L)=0$, hence we may assume that $\kappa(X,L)\geq1$. The lemma follows essentially from the proof of \cite[Theorem 2.1.33]{Laz04}, and we use the notation from that proof. We may assume that $X_\infty$ is smooth and that $X_\infty=X_{(p)}=X_{(q)}$. By possibly blowing up $X_{(m)}$ more, we may assume that all birational maps $\xi_m\colon X_{(m)}\dashrightarrow X_\infty$ are morphisms. 
\[
\xymatrix{ 
X_{(m)} \ar[r]^{\xi_m} & X_\infty \ar[r]^{\phi_\infty} \ar[d]^{u_\infty} & Y_\infty \\
 & X  & 
}
\]
Then it is clear that for every $m$ we have
$$\big|\xi_m^*\big(m(u_\infty^*L+E)\big)\big|=|M_m|+F_m+mE.$$
Since all the maps in the proof of \cite[Theorem 2.1.33]{Laz04} are defined via multiples of $M_m$, it follows that the morphism $\phi_{\infty}\colon X_\infty\to Y_\infty$ is also a model for the Iitaka fibration of $u_\infty^*L+E$, and in particular, for a very general fibre $F$ of $u_\infty$ we have $\kappa\big(F,(u_\infty^*L+E)|_F\big)=0$. Then we set $\pi:=u_\infty$, $Y:=X_\infty$, $f:=\phi_\infty$ and $Z:=Y_\infty$.
\hfill $\Box$
\end{proof}

\subsection{Good models}

A \emph{pair} $(X,\Delta)$ consists of a normal variety $X$ and a Weil $\Q$-divisor $\Delta\geq0$ such that the divisor $K_X+\Delta$ is $\Q$-Cartier. Such a pair is \emph{log smooth} if $X$ is smooth and if the support of $\Delta$ has simple normal crossings. The standard reference for the foundational definitions and results on the singularities of pairs and the Minimal Model Program is \cite{KM98}, and we use these freely in this paper. 

We recall the definition of minimal and good models.

\begin{dfn}{\rm
Let $(X,\Delta) $ be a log canonical pair and let $Y$ be a $\Q$-factorial variety.  A birational contraction $f\colon X\dashrightarrow Y$ is a \emph{minimal model} for $(X,\Delta)$ if $K_Y + f_*\Delta$ is nef, and if there exist resolutions $p\colon W \to X$ and $q\colon W \to Y$ with $q = f \circ p$ 
such that $p^*(K_X + \Delta) =q^*(K_Y + f_*\Delta) +E$, where $E\geq0$ is a $q$-exceptional $\Q$-divisor which contains the proper transform of every $f$-exceptional divisor in its support. If additionally $K_Y + f_*\Delta$ is semiample, the map $f$ is a \emph{good model} for $(X,\Delta)$.}
\end{dfn}

Here we recall additionally that flips for klt pairs exist by \cite{BCHM}. We also use the MMP with scaling of an ample divisor as in \cite[\S3.10]{BCHM}. 

In this context, the following result says, among other things, that if one knows that a good model for a klt pair $(X,\Delta)$ exists, then one knows that there exists also an MMP which leads to a good model; this will be crucial in the proofs in Section \ref{sec:MMPtwist}.

\begin{thm}\label{thm:lai}
Assume the existence of good models for klt pairs in dimensions at most $n-1$. Let $(X,\Delta)$ be a $\Q$-factorial projective klt pair such that $\kappa(X,K_X+\Delta)\geq1$. Then $(X,\Delta)$ has a good model, and every $(K_X+\Delta)$-MMP with scaling of an ample divisor terminates with a good model of $(X,\Delta)$. If $K_X+\Delta$ is additionally nef, then it is semiample.
\end{thm}

\begin{proof}
The result follows by combining \cite[Propositions 2.4 and 2.5, and Theorem 4.4]{Lai11}. Note that \cite[Theorem 4.4]{Lai11} is stated for a terminal variety $X$, but the proof generalises to the context of klt pairs by replacing \cite[Lemma 2.2]{Lai11} with \cite[Lemma 2.10]{HX13}.
\hfill $\Box$
\end{proof}

\subsection{Numerical Kodaira dimension}\label{subsec:numdim}
If $L$ is a pseudoeffective $\R$-Cartier $\R$-divisor on a projective variety $X$, we denote by $P_\sigma(L)$ and $N_\sigma(L)$ the $\R$-divisors forming the Nakayama-Zariski decomposition of $L$, see\ \cite[Chapter III]{Nak04} for the definition and the basic properties. Further, we recall the definition of the numerical Kodaira dimension \cite{Nak04,Kaw85}.

\begin{dfn}\label{dfn:kappa}{\rm
Let $X$ be a smooth projective variety and let $D$ be a pseudoeffective $\Q$-divisor on $X$. If we denote
$$\sigma(D,A)=\sup\big\{k\in\N\mid \liminf_{m\rightarrow\infty}h^0(X, \mathcal O_X(\lfloor ( mD\rfloor+A))/m^k >0\big\}$$
for a Cartier divisor $A$ on $X$, then the {\em numerical dimension\/} of $D$ is
$$\nu(X,D)=\sup\{\sigma(D,A)\mid A\textrm{ is ample}\}.$$
Note that this coincides with various other definitions of the numerical dimension by \cite{Leh13}. If $X$ is a projective variety and if $D$ is a pseudoeffective $\Q$-Cartier $\Q$-divisor on $X$, then we set $\nu(X,D)=\nu(Y,f^*D)$ for any birational morphism $f\colon Y\to X$ from a smooth projective variety $Y$. When the divisor $D$ is nef, then alternatively,
$$\nu(X,D)=\sup\{k\in\N\mid D^k\not\equiv0\}.$$}
\end{dfn}

We use often without explicit mention that the functions $\kappa$ and $\nu$ behave well under proper pullbacks: in other words, if $D$ is a $\Q$-Cartier $\Q$-divisor on a projective variety $X$, and if $f\colon Y\to X$ is a proper surjective morphism, then 
$$\kappa(X,D)=\kappa(Y,f^*D)\quad\text{and}\quad\nu(X,D)=\nu(Y,f^*D).$$ 
If $f$ is birational and $E$ is an effective $f$-exceptional divisor on $Y$, then 
$$\kappa(X,D)=\kappa(Y,f^*D+E)\quad\text{and}\quad\nu(X,D)=\nu(Y,f^*D+E).$$ 
For proofs, see \cite[Lemma II.3.11, Proposition V.2.7(4)]{Nak04}, \cite[Lemma 2.16]{GL13} and \cite[Theorem 6.7]{Leh13}.

We also need the following result \cite[Theorem 4.3]{GL13},  which -- after passing to a $\Q$-factorialisation -- holds without assuming $X$ to be $\Q$-factorial.  

\begin{lem}\label{lem:Kappa=KappaSigma}
Let $(X,\Delta)$ be a klt pair. Then $(X,\Delta)$ has a good model if and only if $\kappa(X,K_X+\Delta)=\nu(X,K_X+\Delta)$.
\end{lem}

The proof of the following simple lemma is analogous to that of \cite[Lemma 3.1]{DL15}.

\begin{lem}\label{lem:pushforward}
Let $\pi\colon X\dashrightarrow Y$ be a birational contraction between projective varieties. Let $D$ be a $\Q$-Cartier $\Q$-divisor on $X$ such that $\pi_*D$ is $\Q$-Cartier. Then $\kappa(Y,\pi_*D)\geq\kappa(X,D)$ and $\nu(Y,\pi_*D)\geq\nu(X,D)$.
\end{lem}

\begin{proof}
By passing to multiples, we may assume that both $D$ and $D'=\pi_*D$ are Cartier. Let $p\colon W \to X$ and $q\colon W \to Y$ be resolutions such that $q = \pi \circ p$. Write
$$p^*D=p_*^{-1}D+E_p^+-E_p^-\quad\text{and}\quad q^*D'=q_*^{-1}D'+E_q^+-E_q^-,$$
where $E_p^+,E_p^-\geq0$ is $p$-exceptional and without common components and $E_q^+,E_q^-\geq0$ is $q$-exceptional and without common components. Then there are $q$-exceptional divisors $E^+,E^-\geq0$ without common components such that $E^--E^+=p_*^{-1}D-q_*^{-1}D'$, and since $\pi$ is a contraction, $E_p^+$ and $E_p^-$ are $q$-exceptional. Therefore,
\begin{align*}
\kappa(Y,D')&=\kappa(W,q^*D'+E_p^+ +E_q^-+E^-)=\kappa(W,p_*^{-1}D+E_p^+ +E_q^++E^+)\\
& \geq\kappa(W,p^*D+E_q^+ +E^+)\geq\kappa(W,p^*D)=\kappa(X,D),
\end{align*}
which was to be proved. The second inequality is analogous.
\hfill $\Box$
\end{proof}

\subsection{Torsion-free and reflexive sheaves}

A coherent sheaf $\mathcal F$ on an algebraic variety $X$ is \emph{reflexive} if the natural homomorphism from $\mathcal F$ to its double dual $\mathcal F^{**}$ is an isomorphism. In particular, a reflexive sheaf $\mathcal F$ is torsion-free. If $X$ is locally factorial, then a reflexive sheaf of rank $1$ is a line bundle \cite[Proposition 1.9]{Har80}. If $X$ is smooth, then a reflexive sheaf is locally free away from a codimension $3$ subset of $X$ \cite[Corollary 1.4]{Har80}, and a torsion-free sheaf is locally free away from a codimension $2$ subset of $X$ \cite[Corollary on p.\ 148]{OSS80}.

Let $\mathcal F$ be a coherent sheaf which is a subsheaf of a locally free sheaf $\mathcal E$. The \emph{saturation} of $\mathcal F$ in $\mathcal E$ is the largest sheaf $\mathcal F'\subseteq\mathcal E$ such that $\mathcal F\subseteq\mathcal F'$, the ranks of $\mathcal F$ and $\mathcal F'$ are the same, and the quotient $\mathcal E/\mathcal F'$ is torsion-free. The saturation $\mathcal F'$ always exists and is a reflexive sheaf, see \cite[Lemma 1.1.16]{OSS80}.

If $\mathcal F$ is a torsion-free coherent sheaf of rank $r$ on a smooth variety $X$, then the \emph{determinant} of $\mathcal F$ is by definition $\det\mathcal F=\big(\bigwedge^r\mathcal F\big)^{**}$. If $\mathcal F$ is reflexive, then this definition coincides with that in \cite[p.\ 129]{Har80}.

Let $\mathcal E$ be a coherent sheaf on an algebraic variety $X$. 
We say that a sheaf $\mathcal F\subseteq\mathcal E$ is \emph{the subsheaf of $\mathcal E$ generated by global sections of $\mathcal E$}
if $\mathcal F$ is the image of the evaluation map
$$ H^0(X,\mathcal E)\otimes\OO_X \to \mathcal E.$$ 

%

The next proposition seems to be folklore, but we include the proof for the benefit of the reader.

\begin{pro} \label{pro:wedge}
Let $X$ be a smooth variety and let $\mathcal F$ be a globally ge\-ne\-ra\-ted torsion-free sheaf on $X$ of rank $r$. If $h^0(X,\mathcal F) \geq r+1$, then 
$$h^0(X, \det \mathcal F) \geq 2.$$
\end{pro} 

\begin{proof}  
Arguing by contradiction, assume that $h^0(X, \det \mathcal F) \leq 1$. Let $X^\circ$ be the largest open subset where $\mathcal F$ is locally free. Then $\codim_X(X\setminus X^\circ)\geq2$, which implies that $h^0(X^\circ, \det \mathcal F|_{X^\circ}) \leq 1$ by \cite[Proposition 1.6]{Har80}. Moreover, since $\mathcal F$ is torsion-free, the restriction map
\begin{equation} \label{eq:restriction}
H^0(X,\mathcal F) \to H^0(X^{\circ} ,\mathcal F \vert_{X^\circ})
\end{equation}
is injective. Let $\eta$ be the generic point of $X$. Our assumption and \eqref{eq:restriction} imply that there are $\omega_1,\dots,\omega_r\in H^0(X^\circ,\mathcal F|_{X^\circ})$ such that $(\omega_1)_\eta,\dots,(\omega_r)_\eta$ are linearly independent in $\OO_{X,\eta}$. Therefore, $\omega_1\wedge\ldots\wedge\omega_r$ defines a non-zero global section of $\det\mathcal F|_{X^\circ}$, thus $h^0(X^\circ,\det\mathcal F|_{X^\circ})=1$. Since $\mathcal F$ is generated by global sections, so is $\det \mathcal F \vert_{X^\circ}$, hence 
$$\det \mathcal F \vert _{X^\circ} \simeq \mathcal O_{X^{\circ}}.$$
But then $\omega_1 \wedge \ldots \wedge \omega_r\in H^0(X^\circ,\det\mathcal F|_{X^\circ})\simeq\C$ is constant on $X^\circ$, and so the sections $\omega_1,\dots,\omega_r$ are linearly independent at every point of $X^\circ$. Therefore, the induced map $\OO_{X^\circ}^{\oplus r}\to\mathcal F|_{X^\circ}$ is an injective vector bundle morphism, hence surjective as the rank of $\mathcal F|_{X^\circ}$ is $r$. In particular, $h^0(X^{\circ} ,\mathcal F \vert_{X^\circ})=r$, which contradicts our assumption and \eqref{eq:restriction}.
\hfill $\Box$
\end{proof} 

The following result is essential for the analysis of subsheaves of sheaves of log differentials. For $\Delta=0$, this is \cite[Theorem 0.1]{CP11}, and in general, this follows from \cite[Theorem 1.2]{CP15} combined with \cite[Theorem 0.2]{BDPP}.

\begin{thm}\label{thm:CP11}
Let $(X,\Delta)$ be a log smooth projective pair, where $\Delta$ is a reduced divisor. Let $\Omega_X^1(\log\Delta)^{\otimes m}\to\mathcal Q$ be a torsion-free coherent quotient for some $m\geq1$. If $K_X+\Delta$ is pseudoeffective, then $c_1(\mathcal Q)$ is pseudoeffective.
\end{thm}

\subsection{Metrics on line bundles}\label{subsec:metric}

Let $X$ be a normal projective variety and let $D$ be a $\Q$-Cartier divisor. Following \cite{DPS01} and \cite{Dem01}, we say that $D$, or $\mathcal O_X(D)$, has a metric with \emph{analytic singularities} and semipositive curvature current, if there exists a positive integer $m$ such that $mD$ is Cartier and if there exists a resolution of singularities $\pi\colon Y \to X$ such that the line bundle $\pi^*\OO_X(mD)$ has a singular metric $h$ whose curvature current is semipositive (as a current), and the local plurisubharmonic weights $\varphi$ of $h$ are of the form
$$ \varphi = \sum \lambda_j \log \vert g_j \vert + O(1),$$
where $\lambda_j$ are positive real numbers, $O(1)$ is a bounded term, and the divisors $D_j$ defined locally by $g_j$ form a simple normal crossing divisor on $Y$. We then have
$$\textstyle\mathcal I(h^{\otimes m})=\OO_Y\big(-\sum\lfloor m\lambda_j\rfloor D_j\big)\quad\text{for every positive integer }m,$$
where $\mathcal I(h^{\otimes m})$ is the multiplier ideal associated to $h^{\otimes m}$. If all $\lambda_j$ are rational, then $h$ has \emph{algebraic singularities}. Further, $\mathcal O_X(D)$  is \emph{hermitian semipositive} if $\pi^*\big(\OO_X(mD)\big)$ has a smooth hermitian metric $h$ whose curvature $\Theta_h(D)$ is semipositive. We mostly use these notions when $D = K_X$. 

\medskip

The following result \cite[Theorem 0.1]{DPS01} is a generalisation of the hard Lefschetz theorem. 

\begin{thm}\label{thm:DPS}
Let $X$ be a compact K\"ahler manifold of dimension $n$ with a K\"ahler form $\omega$. Let $\mathcal L$ be a pseudoeffective line bundle on $X$ with a singular hermitian metric $h$ such that $\Theta_h(\mathcal L) \geq 0$. Then for every nonnegative integer $q$ the morphism
\[
\xymatrix{ 
H^0\big(X,\Omega^{n-q}_X\otimes\mathcal  L\otimes\mathcal I(h)\big) \ar[r]^{\omega^q\wedge\bullet} & H^q\big(X, \Omega^n_X\otimes \mathcal L\otimes\mathcal I(h)\big)
}
\]
is surjective.
\end{thm}

We also need the following result \cite[Lemma 3.6]{LOP16}.

\begin{lem}\label{lem:33}
Let $X$ be a projective manifold and let $L$ be a pseudoeffective Cartier divisor on $X$. Let $h$ be a singular hermitian metric on $\OO_X(L)$ with semipositive curvature current  and multiplier ideal sheaf $\mathcal I(h)$. Let $D$ be an effective Cartier divisor such that $\mathcal I(h)\subseteq\OO_X(-D)$. Then $L-D$ is pseudoeffective.
\end{lem}

\section{MMP with a twist}\label{sec:MMPtwist}

In the Minimal Model Program, starting from a klt pair $(X,\Delta)$ with $K_X+\Delta$ pseudoeffective, one wants to produce a good model for $(X,\Delta)$. It is often difficult even to start the construction due to lack of sections of $K_X+\Delta$. However, if we are in a favourable situation that arbitrarily good approximations of $K_X+\Delta$ have (many) sections, then one can conclude the same for $K_X+\Delta$ itself. That is the content of this section.

We start with the case when $K_X+\Delta$ is nef. It turns out that in this case one can control the growth of sections of $K_X+\Delta$ precisely, and similar techniques will be used in Section \ref{sec:CY}. 

\begin{thm} \label{thm:MMPtwist}
Assume the existence of good models for klt pairs in dimensions at most $n-1$. Let $(X,\Delta)$ be a $\Q$-factorial projective klt pair of dimension $n$ such that $K_X+\Delta$ is nef. Assume that there exist a pseudoeffective $\Q$-divisor $F$ on $X$ and an infinite subset $\mathcal S\subseteq\N$ such that 
\begin{equation}\label{eq:rel2a}
N_m+F\sim_\Q m(K_X+\Delta)
\end{equation}
for all $m\in\mathcal S$, where $N_m\geq0$ are integral Weil divisors. Then 
$$\kappa(X,K_X+\Delta)=\max\{\kappa(X,N_m)\mid m\in\mathcal S\}\geq0.$$
\end{thm} 

\begin{proof}
\emph{Step 1.}
Note first that \eqref{eq:rel2a} implies
\begin{equation}\label{eq:rel2b}
N_p-N_q\sim_\Q (p-q)(K_X+\Delta)\quad\text{for all }p,q\in\mathcal S.
\end{equation}
We claim that for every $m\in\mathcal S$ and every rational $\varepsilon>0$ we have
\begin{equation}\label{eq:kodaira}
\kappa(X,K_X+\Delta+\varepsilon N_m)\geq\kappa(X,N_m).
\end{equation}
Indeed, pick $i_m\in\mathcal S$ very large such that $\varepsilon_m<\varepsilon$, where $\varepsilon_m=\frac{1}{i_m-m}$. Then by \eqref{eq:rel2b} we have
\begin{align}\label{eq:33a}
K_X+\Delta+\varepsilon N_m& =\varepsilon_m\big((i_m-m)(K_X+\Delta)+N_m\big)+(\varepsilon-\varepsilon_m)N_m\\
&\sim_\Q \varepsilon_m N_{i_m}+(\varepsilon-\varepsilon_m)N_m,\notag
\end{align}
which proves \eqref{eq:kodaira}.

\medskip

There are now three cases to consider.

\medskip

\emph{Step 2.}
First assume that $K_X+\Delta+\varepsilon N_p$ is big for some $p\in\mathcal S$ and some rational number $\varepsilon>0$. Then \eqref{eq:rel2a} implies that the divisor
$$(1+\varepsilon p)(K_X+\Delta) \sim_\Q K_X+\Delta+\varepsilon N_p+\varepsilon F$$
is big, and the result is clear since then $N_m$ is big for $m\gg0$ by \eqref{eq:rel2b}. 

\medskip

\emph{Step 3.}
Now assume that $\kappa(X,K_X+\Delta+\varepsilon N_p)=0$ for some $p\in\mathcal S$ and some rational number $\varepsilon>0$. Fix $q\in\mathcal S$ such that $q-p>1/\varepsilon$. Then as in \eqref{eq:33a} we have
$$\textstyle K_X+\Delta+\varepsilon N_p\sim_\Q \frac{1}{q-p} N_q+\big(\varepsilon-\frac{1}{q-p}\big)N_p,$$
hence $\kappa(X,N_q)\leq\kappa(X,K_X+\Delta+\varepsilon N_p)=0$, and therefore $\kappa(X,N_q)=0$. Let $r\in\mathcal S$ be such that $r>q$. Then by \eqref{eq:rel2b} we have
$$K_X+\Delta\sim_\Q\frac{1}{q-p}(N_q-N_p)\quad\text{and}\quad K_X+\Delta\sim_\Q \frac{1}{r-p}(N_r-N_p),$$
so that
$$(r-p)N_q\sim_\Q(q-p)N_r+(r-q)N_p\geq0.$$
Since $\kappa(X,N_q)=0$, this implies
$$(r-p)N_q=(q-p)N_r+(r-q)N_p, $$
and hence $\Supp N_r\subseteq\Supp N_q$ and $\kappa(X,N_r)=0$. Therefore, for $r>q$, all divisors $N_r$ are supported on a reduced Weil divisor. By Lemma \ref{relation}, there are positive integers $k\neq\ell$ larger than $q$ in $\mathcal S$ such that $N_k\leq N_\ell$, and hence by \eqref{eq:rel2b},
$$(\ell-k)(K_X+\Delta)\sim_\Q N_\ell-N_k\geq0,$$
hence $\kappa(X,K_X+\Delta)\geq 0$. Moreover, since then $\kappa(X,K_X+\Delta)\leq\kappa(X,K_X+\Delta+\varepsilon N_p)=0$, we have
$$\kappa(X,K_X+\Delta)=0.$$
If $m$ is an element of $\mathcal S$ with $m\geq q$, then $\kappa(X,N_m)=0$ by above, and if $m<q$, then $0=\kappa(X,N_q)\geq\kappa(X,N_m)$ by \eqref{eq:rel2b}, which yields the result.

\medskip

\emph{Step 4.}
Finally, assume that
\begin{equation}\label{eq:kodaira1}
0<\kappa(X,K_X+\Delta+\varepsilon N_p)<n \quad\text{for every }p\in\mathcal S\text{ and every }\varepsilon>0.
\end{equation}
It suffices to show that
\begin{equation}\label{eq:rel222}
\kappa(X,K_X+\Delta)\geq\kappa(X,N_m)\quad\text{for all large }m\in \mathcal S. 
\end{equation} 
Indeed, then $\kappa(X,K_X+\Delta)\geq0$, hence \eqref{eq:rel2b} gives $\kappa(X,K_X+\Delta)\leq\kappa(X,N_m)$ for all large $m\in \mathcal S$ and $\kappa(X,N_q)\leq\kappa(X,N_p)$ for $p,q\in\mathcal S$ with $q<p$, which implies the theorem.

Fix a positive integer $t$ such that $t(K_X+\Delta)$ is Cartier. Fix integers $\ell>k$ in $\mathcal S$ and fix a rational number $0<\varepsilon\ll1$ such that:
\begin{enumerate}
\item[(a)] the pair $(X,\Delta+\varepsilon N_k)$ is klt, and
\item[(b)] $\varepsilon(\ell-k)>2nt$. 
\end{enumerate}
Denote $\delta=\frac{\varepsilon}{\varepsilon(\ell-k)+1}$ and fix an ample divisor $A$ on $X$. We run the MMP with scaling of $A$ for the klt pair $(X,\Delta+\delta N_k)$. By \eqref{eq:rel2b} we have
\begin{equation}\label{eq:twoMMP}
\textstyle K_X+\Delta+\varepsilon N_\ell\sim_\Q\big(\varepsilon(\ell-k)+1\big)\big(K_X+\Delta+\delta N_k\big),
\end{equation}
hence every step in this MMP is a $(K_X+\Delta+\varepsilon N_\ell)$-negative map. Since we are assuming the existence of good models for klt pairs in lower dimensions, by Theorem \ref{thm:lai} our MMP with scaling of $A$ terminates with a good model for $(X,\Delta+\delta N_k)$. 

We claim that this MMP is $(K_X+\Delta)$-trivial, and hence the proper transform of $t(K_X+\Delta)$ at every step of this MMP is a nef Cartier divisor by \cite[Theorem 3.7(4)]{KM98}. Indeed, it is enough to show the claim for the first step of the MMP, as the rest is analogous. Let $c_R\colon X\to Z$ be the contraction of a $(K_X+\Delta+\delta N_k)$-negative extremal ray $R$ in this MMP. Since by \eqref{eq:rel2b} we have
\begin{equation}\label{eq:rel2c}
K_X+\Delta+\varepsilon N_\ell\sim_\Q K_X+\Delta+\varepsilon N_k+\varepsilon(\ell-k)(K_X+\Delta),
\end{equation}
and as $K_X+\Delta$ is nef, the equation \eqref{eq:twoMMP} implies that $R$ is also $(K_X+\Delta+\varepsilon N_k)$-negative. By the boundedness of extremal rays \cite[Theorem 1]{Kaw91}, there exists a rational curve $C$ contracted by $c_R$ such that 
$$(K_X+\Delta+\varepsilon N_k)\cdot C\geq {-}2n.$$
If $c_R$ were not $(K_X+\Delta)$-trivial, then $t(K_X+\Delta)\cdot C\geq1$ as $t(K_X+\Delta)$ is Cartier. But then \eqref{eq:rel2c} and the condition (b) above yield
$$(K_X+\Delta+\varepsilon N_\ell)\cdot C=(K_X+\Delta+\varepsilon N_k)\cdot C+\varepsilon(\ell-k)(K_X+\Delta)\cdot C>0,$$
a contradiction which proves the claim, i.e.\ the MMP is $(K_X + \Delta)$-trivial. 

\medskip

\emph{Step 5.}
In particular, the numerical Kodaira dimension and the Kodaira dimension of $K_X+\Delta$ are preserved in the MMP, see \cite[Theorem 3.7(4)]{KM98} and \S\ref{subsec:numdim}. Hence, $K_X+\Delta$ is not big by \eqref{eq:kodaira1}. Furthermore, the proper transform of $F$ is pseudoeffective, and the Kodaira dimension of the proper transform of each $N_m$ for $m\in\mathcal S$ does not decrease by Lemma \ref{lem:pushforward}. Therefore, by replacing $X$ by the resulting minimal model, we may assume that $K_X+\Delta+\delta N_k$ is semiample, and hence the divisor $K_X+\Delta+\varepsilon N_\ell$ is also semiample by \eqref{eq:twoMMP}. Note also that $\kappa(X,K_X+\Delta+\varepsilon N_m)>0$ for all $m\in\mathcal S$ by Lemma \ref{lem:pushforward}.

Fix $m\in\mathcal S$ such that $m>\ell$. Then the divisor 
$$K_X+\Delta+\varepsilon N_m\sim_\Q K_X+\Delta+\varepsilon N_\ell+\varepsilon(m-\ell)(K_X+\Delta)$$
is nef. Notice that $K_X+\Delta+\varepsilon N_m$ is not big, since otherwise $K_X + \Delta$ would be big as in Step 2. 

Therefore, we have $0<\kappa(X,K_X+\Delta+\varepsilon N_m)<n$. By \eqref{eq:rel2b} we have
$$\textstyle K_X+\Delta+\varepsilon N_m\sim_\Q\big(\varepsilon(m-k)+1\big)\big(K_X+\Delta+\frac{\varepsilon}{\varepsilon(m-k)+1} N_k\big),$$
and the pair $(X,\Delta+\frac{\varepsilon}{\varepsilon(m-k)+1} N_k)$ is klt. Since we are assuming the existence of good models for klt pairs in lower dimensions, by Theorem \ref{thm:lai} the divisor $K_X+\Delta+\varepsilon N_m$ is semiample. 

Let $\varphi_\ell\colon X\to S_\ell$ and $\varphi_m\colon X\to S_m$ be the Iitaka fibrations associated to $K_X+\Delta+\varepsilon N_\ell$ and $K_X+\Delta+\varepsilon N_m$, respectively. Then there exist ample $\Q$-divisors $A_\ell$ on $S_\ell$ and $A_m$ on $S_m$ such that 
$$K_X+\Delta+\varepsilon N_\ell\sim_\Q\varphi_\ell^*A_\ell\quad\text{and}\quad K_X+\Delta+\varepsilon N_m\sim_\Q\varphi_m^*A_m.$$
If $\xi$ is a curve on $X$ contracted by $\varphi_m$, then 
$$0=(K_X+\Delta+\varepsilon N_m)\cdot \xi=(K_X+\Delta+\varepsilon N_\ell)\cdot \xi+\varepsilon(m-\ell)(K_X+\Delta)\cdot \xi,$$
hence $(K_X+\Delta+\varepsilon N_\ell)\cdot \xi=(K_X+\Delta)\cdot \xi=0$. In particular, $\xi$ is contracted by $\varphi_\ell$, which implies that there exists a morphism $\psi\colon S_m\to S_\ell$ such that $\varphi_\ell=\psi\circ\varphi_m$. 
\[
\xymatrix{ 
& X \ar[ld]_{\varphi_m} \ar[dr]^{\varphi_\ell} & \\
S_m \ar[rr]^{\psi} & & S_\ell 
}
\]
Therefore, denoting $B=\frac{1}{\varepsilon(m-\ell)} (A_m-\psi^*A_\ell)$, we have
$$K_X+\Delta\sim_\Q \frac{1}{\varepsilon(m-\ell)}\big((K_X+\Delta+\varepsilon N_m)-(K_X+\Delta+\varepsilon N_\ell)\big)\sim_\Q\varphi_m^*B.$$
Denoting 
$$\textstyle B_0=\big(m+\frac{1}{\varepsilon}\big)B-\frac{1}{\varepsilon}A_m,$$ 
it is easy to check from \eqref{eq:rel2a} that 
$$F\sim_\Q\varphi_m^*B_0,$$
and hence $B_0$ is pseudoeffective. Therefore the divisor $A_m+B_0$ is big on $S_m$, and
$$(1+\varepsilon m)(K_X+\Delta)\sim_\Q K_X+\Delta+\varepsilon N_m+\varepsilon F\sim_\Q\varphi_m^*(A_m+B_0).$$
By \eqref{eq:kodaira}, this yields
\begin{align}\label{eq:equality0}
\kappa(X,K_X+\Delta)&=\kappa(S_m,A_m+B_0)=\dim S_m\\
&=\kappa(X,K_X+\Delta+\varepsilon N_m)\geq\kappa(X,N_m),\notag
\end{align}
and note that this holds for all $m\in\mathcal S$ sufficiently large. This proves \eqref{eq:rel222} and finishes the proof of the theorem.
\hfill $\Box$
\end{proof}

Now we treat the general case when $K_X+\Delta$ is pseudoeffective. We start with the following result which is implicit already in \cite{DHP13,DL15}. It says that if a klt pair $(X,\Delta)$ has a fibration to a lower dimensional variety which is not a point, then often the existence of good models holds, assuming the Minimal Model Program in lower dimensions.

\begin{pro}\label{pro:contraction}
Assume the existence of good models for klt pairs in dimensions at most $n-1$. Let $(X,\Delta)$ be a $\Q$-factorial projective klt pair such that $K_X+\Delta$ is pseudoeffective. If there exists a fibration $X\to Z$ to a normal projective variety $Z$ such that $\dim Z\geq 1$ and $K_X+\Delta$ is not big over $Z$, then $(X,\Delta)$ has a good model.
\end{pro}

\begin{proof}
The divisor $K_X+\Delta$ is effective over $Z$ by induction on the dimension and by \cite[Lemma 3.2.1]{BCHM}. By \cite[Theorem 2.5]{DL15} and by \cite[Theorem 1.1]{Fuj11} there exists a good model $(X,\Delta)\dashrightarrow (X_{\min},\Delta_{\min})$ of $(X,\Delta)$ over $Z$; alternatively, this follows from \cite[Theorem 2.12]{HX13}. Let $\varphi\colon X_{\min}\to X_\can$ be the corresponding fibration to the canonical model of $(X,\Delta)$ over $Z$.
\[
\xymatrix{ 
X \ar[dr] \ar@{-->}[r] & X_{\min} \ar[d] \ar[r]^{\varphi} & X_\can \ar[dl]\\
& Z & 
}
\]
Since $K_X+\Delta$ is not big over $Z$, we have $\dim X_\can<\dim X$. By \cite[Theorem 0.2]{Amb05a}, there exists a divisor $\Delta_\can\geq0$ on $X_\can$ such that the pair $(X_\can,\Delta_\can)$ is klt and 
$$K_{X_{\min}}+\Delta_{\min}\sim_\Q\varphi^*(K_{X_\can}+\Delta_\can).$$
Since we assume the existence of good models for klt pairs in dimensions at most $n-1$, by Lemma \ref{lem:Kappa=KappaSigma} we have 
$$\kappa(X_\can,K_{X_\can}+\Delta_\can)=\nu(X_\can,K_{X_\can}+\Delta_\can),$$ 
and the result follows by the discussion in \S\ref{subsec:numdim} and by another application of Lemma \ref{lem:Kappa=KappaSigma}.
\hfill $\Box$
\end{proof}

\begin{thm} \label{thm:MMPtwist3}
Assume the existence of good models for klt pairs in dimensions at most $n-1$. Let $(X,\Delta)$ be a $\Q$-factorial projective klt pair of dimension $n$ such that $K_X+\Delta$ is pseudoeffective. Assume that there exist a pseudoeffective $\Q$-divisor $F$ on $X$ and an infinite subset $\mathcal S\subseteq\N$ such that 
\begin{equation}\label{eq:rel2a3}
N_m+F\sim_\Q m(K_X+\Delta)
\end{equation}
for all $m\in\mathcal S$, where $N_m\geq0$ are integral Weil divisors. Then 
$$\kappa(X,K_X+\Delta)=\max\{\kappa(X,N_m)\mid m\in\mathcal S\}\geq0.$$
\end{thm} 

\begin{proof}
We first observe that Steps 1--3 of the proof of Theorem \ref{thm:MMPtwist} work also in the case when $K_X+\Delta$ is pseudoeffective and not only nef. The relation \eqref{eq:rel2a3} implies
\begin{equation}\label{eq:rel2b3}
N_p-N_q\sim_\Q (p-q)(K_X+\Delta)\quad\text{for all }p,q\in\mathcal S,
\end{equation}
and as in the proof of Theorem \ref{thm:MMPtwist}, for every $m\in\mathcal S$ and every rational $\varepsilon>0$ we have
$$\kappa(X,K_X+\Delta+\varepsilon N_m)\geq\kappa(X,N_m).$$
Again as in that proof, we may assume that
\begin{equation}\label{eq:kodaira13}
0<\kappa(X,K_X+\Delta+\varepsilon N_p)<n \quad\text{for every }p\in\mathcal S\text{ and every }\varepsilon>0.
\end{equation}
We first show that
\begin{equation}\label{eq:claim}
\kappa(X,K_X+\Delta)\geq0.
\end{equation} 
Fix $p\in\mathcal S$ and denote $L=K_X+\Delta+N_p$. By Lemma \ref{lem:iitaka} there exists a resolution $\pi\colon Y\to X$ and a morphism $f\colon Y\to Z$:
\[
\xymatrix{ 
Y \ar[d]_{\pi} \ar[r]^{f} & Z\\
X & 
}
\]
such that $\dim Z=\kappa(X,L)\in\{1,\dots,n-1\}$, and for a very general fibre $F$ of $f$ and for every $\pi$-exceptional $\Q$-divisor $G$ on $Y$ we have
\begin{equation}\label{eq:exceptional}
\kappa\big(F,(\pi^*L+G)|_F\big)=0.
\end{equation} 
There exist $\Q$-divisors $\Gamma,E\geq0$ without common components such that
$$K_Y+\Gamma\sim_\Q\pi^*(K_X+\Delta)+E,$$
and it is enough to show that $\kappa(Y,K_Y+\Gamma)\geq0$. By \eqref{eq:exceptional} we have
\begin{equation}\label{eq:Gamma}
\kappa\big(F,(K_Y+\Gamma+\pi^*N_p)|_F\big)=\kappa\big(F,(\pi^*L+E)|_F\big)=0,
\end{equation}
and since $\kappa\big(F,(K_Y+\Gamma)|_F)\geq0$ by induction on the dimension, the equation \eqref{eq:Gamma} implies
$$\kappa\big(F,(K_Y+\Gamma)|_F\big)=0.$$ 
But then $\kappa(Y,K_Y+\Gamma)\geq0$ by Proposition \ref{pro:contraction}, which proves \eqref{eq:claim}.

Note that then by \eqref{eq:rel2b3} we have $\kappa(X,K_X+\Delta)\leq\kappa(X,N_m)$ for all large $m\in \mathcal S$ and $\kappa(X,N_q)\leq\kappa(X,N_p)$ for $p,q\in\mathcal S$ 
with $q<p$, hence it suffices to show that $\kappa(X,K_X+\Delta)\geq\kappa(X,N_m)$ for all $m\in \mathcal S$. Now, by \cite[Theorem 2.5]{DL15} there exists a minimal model 
$$\varphi\colon (X,\Delta)\dashrightarrow (X_{\min},\Delta_{\min})$$ 
of $(X,\Delta)$, and for all $m\in\mathcal S$ we have
$$\varphi_*N_m+\varphi_*F\sim_\Q m(K_{X_{\min}}+\Delta_{\min}).$$
Then Theorem \ref{thm:MMPtwist} implies 
$$\kappa(X_{\min},K_{X_{\min}}+\Delta_{\min})=\max\{\kappa(X_{\min},\varphi_*N_m)\mid m\in\mathcal S\},$$ 
and the result follows from Lemma \ref{lem:pushforward}.
\hfill $\Box$
\end{proof}

Finally, the following is an immediate corollary of the proof of Theorem \ref{thm:MMPtwist3}. 

\begin{thm} \label{thm:MMPtwist2}
Assume the existence of good models for klt pairs in dimensions at most $n-1$. Let $(X,\Delta)$ be a $\Q$-factorial projective klt pair of dimension $n$ such that $K_X+\Delta$ is pseudoeffective. Assume that there exists a pseudoeffective $\Q$-divisor $F$ on $X$ such that $\kappa\big(X,m(K_X+\Delta)-F\big)\geq1$ for infinitely many $m$. Then $(X,\Delta)$ has a good model.
\end{thm} 

\begin{proof}
By assumption, the set $\mathcal S=\{m\in\N\mid \kappa\big(X,m(K_X+\Delta)-F\big)\geq1\}$ has infinitely many elements, and for every $m\in\mathcal S$, there exists a $\Q$-divisor $N_m\geq0$ with $\kappa(X,N_m) \geq 1$ such that
$$N_m+F\sim_\Q m(K_X+\Delta).$$
Then as in the proof of Theorem \ref{thm:MMPtwist3}, we obtain 
$$\kappa(X,K_X+\Delta)\geq\max\{\kappa(X,N_m)\mid m\in\mathcal S\}\geq1;$$
note that the condition that the divisors $N_m$ are \emph{integral} was only used in Step 3 of the proof of Theorem \ref{thm:MMPtwist}, and the situation in this step  does not happen in our context by the equation \eqref{eq:33a}. We conclude by Theorem \ref{thm:lai}.
\hfill $\Box$
\end{proof}

\section{Nonvanishing}\label{sec:thmA}

In this section, we prove part (i) of  Theorem \ref{thm:A}. Note that by \cite[Theorem 8.8]{DHP13}, assuming the existence of good models for a klt pairs in dimensions at most $n-1$, the nonvanishing for \emph{klt pairs} in dimension $n$ reduces to nonvanishing for \emph{terminal varieties} in dimension $n$. Therefore, one does not gain any generality when one considers nonvanishing for pairs. The situation is, however, different when one considers semiampleness.

\medskip

We start with the following two results which will be used several times in the paper. 

\begin{lem} \label{lemmafund} 
Let $X$ be a projective manifold and let $\mathcal E$ be a locally free sheaf on $X$. Let $\mathcal L$ be a pseudoeffective line bundle on $X$ which is not numerically trivial, and assume furthermore that 
$$ H^0(X,\mathcal E \otimes \mathcal L^{\otimes m}) \neq 0\quad\text{for infinitely many integers }m.$$
Then there exist a positive integer $r$ and a saturated line bundle $\mathcal M$ in $\bigwedge^r\mathcal E$ such that 
$$H^0(X,\mathcal M\otimes \mathcal L^{\otimes m}) \neq 0\quad\text{for infinitely many positive integers }m.$$
\end{lem} 

\begin{proof} 
By assumption, there exists an infinite set $\mathcal T\subseteq \Z$ such that 
$$ H^0(X, \mathcal E \otimes \mathcal L^{\otimes m}) \neq 0\quad\text{for all }m\in\mathcal T.$$
Denote $Y=\PS(\mathcal E)$ with the projection $f\colon Y\to X$. First note that 
$$H^0(X,\mathcal E\otimes \mathcal L^{\otimes m})\simeq H^0(Y,\OO_Y(1)\otimes f^*\mathcal L^{\otimes m}).$$
Since $\mathcal L$ is pseudoeffective and not numerically trivial, the line bundle $\OO_Y(1)\otimes f^*\mathcal  L^{\otimes m}$ is not pseudoeffective for $m\ll0$, hence there are only finitely many negative integers in $\mathcal{T}$. Therefore, we may assume that $\mathcal T\subseteq\N$.

Fix a nontrivial section of $H^0(X,\mathcal E \otimes \mathcal L^{\otimes m})$ for every $m \in \mathcal T$. It gives an inclusion 
$$\mathcal L^{\otimes -m} \to \mathcal E,$$  
and let $\mathcal F \subseteq\mathcal E$ be the image of the induced map 
$$\textstyle\bigoplus_{m\in \mathcal T} \mathcal L^{\otimes -m} \to \mathcal E.$$
Then $\mathcal F$ is quasi-coherent by \cite[Proposition II.5.7]{Har77}, and therefore a torsion-free coherent sheaf as it is a subsheaf of the torsion-free coherent sheaf $\mathcal E$. 
Let $r$ be the rank of $\mathcal F$. We may assume that there exist infinitely many $r$-tuples $(m_1,\dots,m_r)$ such that the image of the map 
\begin{equation}\label{eq:inclusion}
\mathcal L^{\otimes -m_1} \oplus\cdots\oplus \mathcal L^{\otimes -m_r} \to\mathcal F
\end{equation}
has rank $r$: indeed, if this is not the case, we replace $\mathcal T$ by a suitable infinite subset, and the rank of $\mathcal F$ is smaller than $r$. Taking determinants in \eqref{eq:inclusion} yields inclusions
\begin{equation}\label{eq:inclusion2}
\mathcal L^{\otimes {-}(m_1+\dots+m_r)} \to \det\mathcal F\subseteq \bigwedge^r\mathcal E.
\end{equation}
Let $\mathcal M$ be the line bundle which is the saturation of $\det\mathcal F$ in $\bigwedge^r\mathcal E$. Then by \eqref{eq:inclusion2} there exists an infinite set $\mathcal S\subseteq \N$ such that
$$H^0(X,\mathcal M\otimes \mathcal L^{\otimes m}) \ne 0\quad\text{for all }m\in\mathcal S,$$
which proves the result.
%
\hfill $\Box$
\end{proof} 

Now we use Theorem \ref{thm:CP11} to connect the previous lemma to our setting.

\begin{pro}\label{pro:quotient}
Let $(X,\Delta)$ be a log smooth projective pair, where $\Delta$ is a reduced divisor and $K_X+\Delta$ is pseudoeffective. Let $\mathcal E$ be a tensor representation of $\Omega_X^1(\log\Delta)$, let $\mathcal M$ be a saturated line bundle in $\mathcal E$, and let $L$ be an integral divisor on $X$. Assume that there exists an infinite set $\mathcal S\subseteq\N$ and integral divisors $N_m\geq0$ for $m\in\mathcal S$ such that
\begin{equation}\label{eq:infmanym}
\OO_X(N_m)\simeq \mathcal M\otimes \OO_X(mL) \quad\text{for every }m\in\mathcal S.
\end{equation}
Then there exist a positive integer $\ell$ and a pseudoeffective divisor $F$ such that
$$N_m+ F \sim mL+\ell (K_X+\Delta) \quad\text{for every }m\in\mathcal S.$$
\end{pro}

\begin{proof}
Consider the exact sequence
$$ 0 \to \mathcal M \to \mathcal E\to \mathcal Q \to 0.$$
Since $\mathcal M$ is saturated in $\mathcal E$, the sheaf $\mathcal Q$ is torsion-free, and hence $F :=c_1(\mathcal Q)$ is pseudoeffective by Theorem \ref{thm:CP11}. From the above exact sequence, there exists a positive integer $\ell$ such that
\begin{equation}\label{eq:ell}
\OO_X\big(\ell(K_Y+\Delta)\big)\simeq \OO_Y(F)\otimes\mathcal M.
\end{equation}
The result follows from \eqref{eq:infmanym} and \eqref{eq:ell}.
\hfill $\Box$
\end{proof}

The following result implies Theorem \ref{thm:A}(i), since any tensor representation of a vector bundle can be embedded as a submodule in some high tensor power, see \cite[Chapter III, \S6.3 and \S7.4]{Bou98}.

\begin{thm} \label{thm:nonvanishingForms}
Assume the existence of good models for klt pairs in dimensions at most $n-1$. Let $X$ be a projective terminal variety of dimension $n$ with $K_X$ pseudoeffective. 
Let $\pi\colon Y\to X$ be a resolution of $X$. Assume that for some positive integer $p$ we have  
$$ H^0\big(Y,(\Omega^1_Y)^{\otimes p} \otimes \OO_Y(m\pi^*K_X)\big) \neq 0$$
for infinitely many $m$ such that $mK_X$ is Cartier. Then $\kappa (X,K_X) \geq 0$. 
\end{thm} 

\begin{proof} 
Let $\rho\colon X'\to X$ be a $\Q$-factorialisation of $X$, see \cite[Corollary 1.37]{Kol13}. Then $\rho$ is an isomorphism in codimension $1$, hence $K_{X'}=\rho^*K_X$ and $X'$ is terminal. By replacing $X$ by $X'$, we may thus assume that $X$ is $\Q$-factorial. 
We first note that we have $K_X \not \equiv 0$, since otherwise $\kappa (X,K_X) = 0$ by \cite[Theorem 8.2]{Kaw85b}.

Let $\pi\colon Y\to X$ be a resolution of $X$. We apply Lemma \ref{lemmafund} with $\mathcal E = (\Omega^1_Y)^{\otimes p} $ and $\mathcal L = \pi^*\OO_X(m_0K_X)$, where $m_0$ is the smallest positive integer such that $m_0K_X$ is Cartier. Then there exist a positive integer $r$, 
a saturated line bundle $\mathcal M$ in $\bigwedge^r\mathcal E$, an infinite set $\mathcal S\subseteq\N$ and integral divisors $N_m\geq0$ for $m\in\mathcal S$ such that 
\begin{equation}\label{eq:infmany}
\OO_Y(N_m) \simeq \mathcal M\otimes \mathcal L^{\otimes m}.
\end{equation}
Then Proposition \ref{pro:quotient} implies that there exist a positive integer $\ell$ and a pseudoeffective divisor $F$ such that
\begin{equation}\label{eq:relation}
N_m+ F \sim mm_0\pi^*K_X+\ell K_Y.
\end{equation}
Noting that $\pi_*N_m$ is effective and that $\pi_*F$ is pseudoeffective, by pushing forward the relation \eqref{eq:relation} to $X$ we get
\begin{equation}\label{eq:rel2}
\pi_*N_m+\pi_*F\sim_\Q (mm_0+\ell)K_X.
\end{equation}
Now Theorem \ref{thm:MMPtwist3} gives a contradiction.
\hfill $\Box$
\end{proof} 

The same proof also shows the following variant. 

\begin{thm} \label{thm:nonvanishingForms2}
Assume the existence of good models for klt pairs in dimensions at most $n-1$. Let $X$ be a projective terminal variety of dimension $n$ with $K_X$ pseudoeffective. Assume that $\kappa(X,K_X) = - \infty$. Let $\pi\colon Y\to X$ be a resolution of $X$, and let $E_1,\dots,E_\ell$ be all $\pi$-exceptional prime divisors on $Y$. Then for all integers $\lambda _i$, for all $m$ sufficiently large such that $mK_X$ is Cartier and for all $q\geq0$ we have
$$\textstyle H^0\big(Y,\Omega^q_Y(\log\sum E_i) \otimes \OO_Y(m(\pi^*K_X+\sum\lambda_i E_i))\big)=0.$$
In particular,
$$\textstyle H^0\big(Y,\Omega^q_Y(\log\sum E_i) \otimes \OO_Y(mK_Y)\big)=0$$
for $m$ sufficiently large such that $mK_X$ is Cartier and $q\geq0$.
\end{thm} 

Now we are ready to state the first corollary of Theorem \ref{thm:nonvanishingForms}. The details are taken from \cite[Theorem 2.14]{DPS01}, but we include the proof for the benefit of the reader.

\begin{cor}\label{cor:nv}
Assume the existence of good models for klt pairs in dimensions at most $n-1$. Let $X$ be a projective terminal variety of dimension $n$ with $K_X$ pseudoeffective. Suppose that $K_X$ has a metric with algebraic singularities and semipositive curvature current. If $\chi(X,\OO_X)\neq0$, then $\kappa(X,K_X) \geq 0$. 
\end{cor}

\begin{proof} 
Again, we may assume that $X$ is $\Q$-factorial. 
Choose a resolution $\pi\colon Y\to X$ such that for some positive integer $\ell$ the divisors $\ell K_X$ and $\ell K_Y$ are Cartier, and there exists a metric $h$ with algebraic singularities on $\pi^*\OO_X(\ell K_X)$ as in \S\ref{subsec:metric}. Then the local plurisubharmonic weights $\varphi$ of $h$ are of the form
$$\varphi = \sum_{j=1}^r \lambda_j \log \vert g_j \vert + O(1),$$
where $\lambda_j$ are positive rational numbers and the divisors $D_j$ defined locally by $g_j$ form a simple normal crossing divisor on $Y$. We have
\begin{equation}\label{eq:metric1}
\textstyle\mathcal I(h^{\otimes m})=\OO_Y\big(-\sum_{j=1}^r\lfloor m\lambda_j\rfloor D_j\big).
\end{equation}

Assume that $\kappa(X,K_X)=-\infty$. Then by Theorem \ref{thm:nonvanishingForms}, for all $p\geq 0$ and for all $m$ sufficiently divisible we have
$$ H^0\big(Y,\Omega^p_Y \otimes \pi^*\OO_X(m\ell K_X)\big)=0,$$
and thus
$$H^0\big(Y,\Omega^p_Y\otimes \pi^*\OO_X(m\ell K_X)\otimes\mathcal I(h^{\otimes m})\big) = 0.$$
Theorem \ref{thm:DPS} implies that for all $p\geq 0$ and for all $m>0$ sufficiently divisible,
$$H^p\big(Y,\OO_Y(K_Y+m\ell\pi^* K_X)\otimes\mathcal I(h^{\otimes m})\big) = 0,$$
which together with \eqref{eq:metric1} and Serre duality yields 
\begin{equation}\label{eq:euler}
\textstyle\chi\big(Y,\OO_Y\big(\sum_{j=1}^r\lfloor m\lambda_j\rfloor D_j-m\ell\pi^* K_X\big)\big) = 0
\end{equation}
for all $m>0$ sufficiently divisible. There are integers $p_j$ and $q_j\neq0$ such that $\lambda_j=p_j/q_j$, and denote $q=\prod q_j$ and $D=\sum \frac{qp_j}{q_j} D_j-q\ell\pi^* K_X$. Then \eqref{eq:euler} implies
$$\chi(Y,\OO_Y(mD)) = 0\quad \text{for $m>0$ sufficiently divisible}.$$
But then $\chi(Y,\OO_Y(mD)) = 0$ for all integers $m$, and hence $\chi(Y,\OO_Y) = 0$. Since $X$ has rational singularities, this implies $\chi(X,\OO_X) = 0$, a contradiction which finishes the proof.
\hfill $\Box$
\end{proof} 

When $K_X$ is hermitian semipositive, the conclusion is much stronger.

\begin{cor}\label{cor:semipositive}
Assume the existence of good models for klt pairs in dimensions at most $n-1$. Let $X$ be a projective terminal variety of dimension $n$. Assume that there is a positive integer $m$ and a desingularisation $\pi\colon Y \to X$ such that $\pi^*\mathcal O_X(\ell K_X)$ has a singular metric $h$ with semipositive curvature current and vanishing Lelong numbers. If the numerical polynomial 
$$ P(m) = \chi\big(X,\mathcal O_X(m\ell K_X)\big)$$
is not identically zero, then $K_X$ is semiample. In particular, the result holds when $K_X$ is hermitian semipositive.
\end{cor}

\begin{proof} 
We follow closely the proof of Corollary \ref{cor:nv}. We may assume that $X$ is $\Q$-factorial. By \cite[Theorem 5.1]{GM17}, it suffices to show that $\kappa (X,K_X) \geq 0$.
 
Arguing by contradiction, assume that $\kappa(X,K_X) = - \infty$. There exists a resolution $\pi\colon Y\to X$ such that for some positive integer $\ell$ the divisors $\ell K_X$ and $\ell K_Y$ are Cartier, and there exists a smooth hermitian semipositive metric $h$ on $\pi^*\OO_X(\ell K_X)$. Note that 
$$ \mathcal I(h^{\otimes m}) = \mathcal O_Y \quad\text{for all positive integers $m$.}$$
Then by Theorem \ref{thm:nonvanishingForms}, for all $p\geq 0$ and for all $m>0$ sufficiently divisible we have
$$ H^0\big(Y,\Omega^p_Y \otimes \pi^*\OO_X(m\ell K_X)\big)=0,$$
which together with Theorem \ref{thm:DPS} and Serre duality implies 
$$H^{n-p}\big(Y,\pi^*\OO_X(-m\ell K_X)\big) = 0,$$
hence $\chi\big(Y,\pi^*\OO_X(m\ell K_X)\big) = 0$ for all $m$. Since $X$ has rational singularities, we deduce
$$ \chi\big(X,\OO_X(m\ell K_X)\big) = 0\quad\text{for all $m$},$$
a contradiction. 
\hfill $\Box$
\end{proof}

\begin{rem}{\rm
The proof of Corollary \ref{cor:semipositive} actually gives more: if $K_X$ is not semiample, then for all $q$ and all $m$ sufficiently divisible we have
$$ H^q(X,\mathcal O_X(-m K_X)) = 0.$$}
\end{rem} 

\section{Semiampleness}\label{sec:thmC}

In this section we prove Theorem \ref{thm:C}(i) as a corollary of the following more general result.

\begin{thm} \label{thm:abundanceForms}
Assume the existence of good models for klt pairs in dimensions at most $n-1$. Let $(X,\Delta)$ be a  projective terminal pair of dimension $n$ such that $\kappa(X,K_X)\geq0$. Let $\pi\colon Y\to X$ be a sufficiently high resolution of $X$, fix a tensor representation $\mathcal E$ of $\Omega_Y^1(\log\lceil\pi^{-1}_*\Delta\rceil)$, and let $q$ be the rank of $\mathcal E$. 
If there exist $m_0 \geq 0$ such that $m_0(K_X+\Delta)$ is Cartier and such that
$$h^0\big(Y,\mathcal E\otimes \OO_Y(m_0\pi^*(K_X+\Delta))\big)\geq q+1,$$
then $(X,\Delta)$ has a good model.
\end{thm} 

\begin{proof} 
\emph{Step 1.}
We argue by contradiction, i.e.\ we assume that $(X,\Delta)$ does not have a good model. By passing to a $\Q$-factorialisation, we may assume that $X$ is $\Q$-factorial. Since $\kappa(X,K_X)\geq0$, there exists a $\Q$-divisor $D\geq0$ such that $K_X\sim_\Q D$. For a rational number $0<\varepsilon\ll1$, denote $\Delta'=(1+\varepsilon)\Delta+\varepsilon D$ and $D'=(1+\varepsilon)(D+\Delta)$, so that $\lfloor\Delta'\rfloor=0$, $D'\geq0$, $\Supp\Delta'=\Supp D'$, the pair $(X,\Delta')$ is terminal, and we have 
$$(1+\varepsilon)(K_X+\Delta)\sim_\Q K_X+\Delta'\sim_\Q D'.$$
Let $\pi\colon Y\to X$ be a log resolution of $(X,\Delta')$, and write 
$$K_Y+\Gamma=\pi^*(K_X+\Delta')+E\sim_\Q \pi^*D'+E,$$
where $\Gamma$ and $E$ are effective and have no common components. Then the pair $(Y,\Gamma)$ does not have a good model by \cite[Lemma 2.10]{HX13}. Note that $\Supp\Gamma\subseteq\Supp(\pi^*D'+E)$ since $(X,\Delta')$ is terminal, and let $\Gamma_Y=(\pi^*D'+E)_{\textrm{red}}$. Denoting $D_Y=\pi^*D'+E+\Gamma_Y-\Gamma\geq0$, we have
\begin{equation}\label{eq:support}
K_Y+\Gamma_Y\sim_\Q D_Y\quad\text{and}\quad \Supp \Gamma_Y=\Supp D_Y=\Supp(\pi^*D'+E).
\end{equation}
Then the pair $(Y,\Gamma_Y)$ is dlt, and $\kappa(Y,K_Y+\Gamma_Y)=\kappa(Y,K_Y+\Gamma)$ and $\nu(Y,K_Y+\Gamma_Y)=\nu(Y,K_Y+\Gamma)$ by \cite[Lemma 2.9]{DL15}, hence 
$$\kappa(Y,K_Y+\Gamma_Y)\neq\nu(Y,K_Y+\Gamma_Y)$$ 
by Lemma \ref{lem:Kappa=KappaSigma}. Let $\mathcal E'$ be a tensor representation of $\Omega_Y^1(\log\Gamma_Y)$ corresponding to $\mathcal E$. Observing that $\lceil\pi^*\Delta\rceil\leq\Gamma_Y$, we have an inclusion
$$\mathcal E\otimes \OO_Y(m_0\pi^*(K_X+\Delta))\to \mathcal E' \otimes \OO_Y(m_0(K_Y+\Gamma_Y)),$$
hence 
$$ h^0\big(Y,\mathcal E' \otimes \OO_Y(m_0(K_Y+\Gamma_Y))\big)\geq q+1.$$

\medskip

\emph{Step 2.}
Let $\mathcal F$ be the subsheaf of $\mathcal E' \otimes \OO_Y(m_0(K_Y+\Gamma_Y))$ generated by its global sections, and let $r$ be the rank of $\mathcal F$. Then
$$\det\mathcal F\subseteq\bigwedge^r\mathcal E' \otimes\OO_Y(rm_0(K_Y+\Gamma_Y)),$$
and there exists a Cartier divisor $N$ such that $\OO_Y(N)$ is the saturation of $\det\mathcal F$ in $\bigwedge^r\mathcal E'\otimes\OO_Y(rm_0(K_Y+\Gamma_Y))$. By Proposition \ref{pro:wedge} we have
\begin{equation}\label{eq:infmany2}
h^0(Y,N)\geq2,
\end{equation}
and denote 
\begin{equation}\label{eq:18}
{-}F_Y=N-rm_0(K_Y+\Gamma_Y).
\end{equation}
Then we have the exact sequence
$$ 0 \to \OO_Y(-F_Y) \to \bigwedge^r\mathcal E' \to \mathcal Q \to 0.$$
Since $\OO_Y(-F_Y)$ is saturated in $\bigwedge^r\mathcal E'$, the sheaf $\mathcal Q$ is torsion-free, and hence $F=c_1(\mathcal Q)$ is pseudoeffective by Theorem \ref{thm:CP11}. From the above exact sequence, there exists a positive integer $\ell$ such that $\ell (K_Y+\Gamma_Y)\sim F-F_Y$. 

Let $d$ be the smallest positive integer such that $H^0(Y,d(K_Y+\Gamma_Y))\neq0$. Denote $\mathcal S=\{rm_0+id\mid i\in\N\}\subseteq\Z$ and 
$$N_{m+\ell}=N+(m-rm_0)(K_Y+\Gamma_Y)\quad\text{for }m\in\mathcal S.$$ 
From \eqref{eq:infmany2}, for every $m\in\mathcal S$ we have $\kappa(Y,N_{m+\ell})\geq1$, and \eqref{eq:18} gives
$$N_{m+\ell}+F\sim (m+\ell) (K_Y+\Gamma_Y).$$
For a rational number $0<\delta\ll1$, denote $\Gamma_Y'=\Gamma_Y-\varepsilon D_Y$, and note that $\lfloor\Gamma_Y'\rfloor=0$ by \eqref{eq:support}. Therefore, the pair $(Y,\Gamma_Y')$ is klt by \cite[Proposition 2.41]{KM98}. We have $(1-\delta)(K_Y+\Gamma_Y)\sim_\Q K_Y+\Gamma_Y'$ by \eqref{eq:support}, and therefore
$$\textstyle\frac{1}{1-\delta} N_{m+\ell}+\frac{1}{1-\delta} F\sim_\Q (m+\ell)(K_Y+\Gamma_Y').$$
Now Theorem \ref{thm:MMPtwist2} gives a contradiction. 
\hfill $\Box$
\end{proof} 

\begin{cor}\label{cor:chi}
Assume the existence of good models for klt pairs in dimensions at most $n-1$. Let $(X,\Delta)$ be a $\Q$-factorial projective terminal pair of dimension $n$ such that $\kappa(X,K_X)\geq0$. 
Suppose that 
$$ h^q(X,\OO_X) > \binom{n}{q}\quad\text{for some }q.$$
Then $(X,\Delta)$ has a good model. In particular, if $\vert\chi(X,\OO_X)\vert > 2^{n-1}$, then $(X,\Delta)$ has a good model. 
\end{cor}

\begin{proof}
Let $\pi\colon Y\to X$ be a sufficiently high resolution of $X$. Applying Theorem \ref{thm:abundanceForms} with $m = 0$ we obtain
$$ h^0\big(Y,\Omega^q_Y) > \binom{n}{q}\quad\text{for all }q,$$ 
hence the first statement follows from Hodge symmetry, since $X$ has rational singularities.
\hfill $\Box$
\end{proof}

\section{Numerical dimension 1}\label{sec:nd1}

In this section we show that some of the previous results hold unconditionally in every dimension, if one assumes that the numerical dimension of the canonical class is $1$. The following is the key technical observation.

\begin{thm} \label{thm:nu1a}
Let $X$ be a projective $\Q$-factorial variety of dimension $n$, and let $L$ be a nef divisor on $X$ such that $\nu(X,L)=1$. Assume that there exist a pseudoeffective $\Q$-divisor $F$ and a non-zero $\Q$-divisor $D \geq 0$ on $X$ such that 
$$D+F\sim_\Q L.$$
Then there exists a $\Q$-divisor $E\geq0$ such that 
$$L\equiv E\quad\text{and}\quad \kappa(X,E)\geq\kappa(X,D).$$
\end{thm} 

\begin{proof}
Let $f\colon Y\to X$ be a resolution of $X$, and denote $L'=f^*L$, $D'=f^*D$ and $F'=f^*F$, so that $D'+F'\sim_\Q L'$. Let $P=P_\sigma(F')$ and $N=N_\sigma(F')\geq0$, so that we have the Nakayama-Zariski decomposition
$$F' = P + N. $$

Assume first that $P \not\equiv 0$. Let $S$ be a surface in $Y$ cut out by $n-2$ general hyperplane sections. Then $P|_S$ is nef by \cite[Remark III.2.8 and paragraph after Corollary V.1.5]{Nak04}, and in particular
\begin{equation}\label{eq:restrictionNef}
(P|_S)^2\geq0.
\end{equation}
On the other hand, since $\nu(Y,L')=1$, we have
$$0=(L'|_S)^2=L'|_S\cdot P|_S+L'|_S\cdot N|_S+L'|_S\cdot D'|_S,$$
hence 
$$L'|_S\cdot P|_S=L'|_S\cdot N|_S=L'|_S\cdot D'|_S=0.$$
Now the Hodge index theorem implies $(P|_S)^2\leq0$, and hence $(P|_S)^2=0$ by \eqref{eq:restrictionNef}. Then Lemma \ref{lem:hodge} yields $P|_S\equiv\lambda L'|_S$ for some real number $\lambda > 0$, and hence $P\equiv \lambda L'$ by the Lefschetz hyperplane section theorem. Note that $D' \neq 0$ implies $\lambda < 1$. Therefore, setting 
$$E'=\frac{1}{1-\lambda}(N+D')-\varepsilon D'\geq0$$
for a rational number $0<\varepsilon\ll1$, we obtain
$$L'-\varepsilon D'\equiv E'.$$
Let $E_1,\dots,E_r$ be the components of $E'$ and let $\pi\colon \Div_\R(Y)\to N^1(Y)_\R$ be the standard projection. Then $\pi^{-1}\big(\pi(L'-\varepsilon D')\big)\cap\sum\R_+E_i$ is a rational affine subspace of $\sum\R E_i\subseteq\Div(Y)_\R$ which contains $E'$, hence there exists a rational point 
$$0\leq E''\in \pi^{-1}\big(\pi(L'-\varepsilon D')\big)\cap\sum\R_+E_i.$$ 
Setting $E=f_*(E''+\varepsilon D')$, we have $L\equiv E$ and $E\geq\varepsilon D$, which proves the result in the case $P\not\equiv0$.

If $P\equiv 0$, denote $E'=N+(1-\varepsilon) D'\geq0$ for a rational number $0<\varepsilon\ll1$, so that $L'-\varepsilon D'\equiv E'$. We conclude as above.
\hfill $\Box$
\end{proof}

\begin{cor}\label{cor:nu1}
Let $(X,\Delta)$ be a $\Q$-factorial projective klt pair such that $K_X+\Delta$ is nef and $\nu(X,K_X+\Delta)=1$. Assume that there exist a pseudoeffective $\Q$-divisor $F$ and a non-zero $\Q$-divisor $D \geq 0$ on $X$ such that $K_X+\Delta\sim_\Q D+F$. Then $\kappa(X,K_X+\Delta)\geq\kappa(X,D)$.
\end{cor}

\begin{proof}
By Theorem \ref{thm:nu1a} applied to $L=K_X+\Delta$, there exists an effective $\Q$-divisor $E$ on $X$ such that $K_X+\Delta\equiv E$ and $\kappa(X,E)\geq\kappa(X,D)$. By \cite[Theorem 0.1]{CKP12} we have $\kappa(X,K_X+\Delta)\geq\kappa(X,E)$, and the result follows.
\hfill $\Box$
\end{proof}

Theorem \ref{thm:A}(ii) is now a special case of:

\begin{thm} \label{thm:nonvanishingFormsnu1}
Let $X$ be a minimal projective terminal variety of dimension $n$. Assume that $\nu(X,K_X)=1$.
Let $\pi\colon Y\to X$ be a resolution of $X$ and assume that there is a positive integer $p$ such that 
$$ H^0\big(Y,(\Omega^1_Y)^{\otimes p} \otimes \OO_Y(m\pi^*K_X)\big) \ne 0$$
for infinitely many $m$ such that $mK_X$ is Cartier. Then $\kappa(X,K_X) \geq 0$.
\end{thm} 

\begin{proof}
The proof is the same as the proof of Theorem \ref{thm:nonvanishingForms}, by invoking Corollary \ref{cor:nu1} instead of Theorem \ref{thm:MMPtwist3}.
\hfill $\Box$
\end{proof}

\begin{cor}\label{cor:nd1}
Let $X$ be a minimal  projective terminal variety of dimension $n$ such that $\nu(X,K_X) = 1$. Assume that $K_X$ has a singular metric with algebraic singularities and semipositive curvature current. If $\chi(X,\OO_X)\neq0$, then $\kappa(X,K_X)\geq0$. In particular, the result holds if $K_X$ is hermitian semipositive.
\end{cor}

\begin{proof}
The proof is the same as that of Corollary \ref{cor:nv}, by invoking Theorem \ref{thm:nonvanishingFormsnu1} instead of Theorem \ref{thm:nonvanishingForms}.
\hfill $\Box$
\end{proof}

\begin{thm} \label{thm:multiplier2} 
Let $X$ be a $\Q$-factorial projective variety and let $L$ be a nef divisor on $X$ with $\nu(X,L)=1$. Assume that there exists a resolution $\pi\colon Y\to X$ and a singular metric $h$ on $\pi^*\OO_X(L)$ with semipositive curvature current such that the multiplier ideal sheaf $\mathcal I(h)$ is different from $\OO_Y$. Then there exists a $\Q$-divisor $D\geq0$ such that 
$$L\equiv D\quad\text{and}\quad \kappa(X,D)\geq0.$$
\end{thm}

\begin{proof}
Let $V \subseteq Y$ be the subspace defined by $\mathcal I(h)$, and let $y$ be a closed point in $V$ with ideal sheaf $\mathcal I_y$ in $y$. Let $\mu\colon \widehat Y \to Y$ be the blow-up of $Y$ at $y$ and let $E = \pi^{-1}(y) $ be the exceptional divisor. Let $\widehat h$ be the induced metric on $(\pi\circ\mu)^*\OO_X(L)$. By \cite[Proposition 14.3]{Dem01}, we have
$$\mathcal I(\widehat h) \subseteq \mu^{-1}\mathcal I(h)\cdot\OO_{\widehat Y}\subseteq\mu^{-1}\mathcal I_y\cdot\OO_{\widehat Y}=\OO_{\widehat Y}(-E).$$ 
By Lemma \ref{lem:33}, the divisor $(\pi\circ\mu)^*L - E$ is pseudoeffective. Then by Lemma \ref{thm:nu1a} there exists a $\Q$-divisor $\widehat D\geq0$ on $\widehat Y$ such that $(\pi\circ\mu)^*L\equiv\widehat D$, and we set $D=(\pi\circ\mu)_*\widehat D$. 
\hfill $\Box$
\end{proof}

\begin{cor}\label{cor:multiplier}
Let $(X,\Delta)$ be a projective klt pair such that $K_X+\Delta$ is nef and $\nu(X,K_X+\Delta)=1$. Assume that there exist a resolution $\pi\colon Y\to X$, a positive integer $m$ such that $m(K_X+\Delta)$ is Cartier, and a singular metric $h$ on $\pi^*\OO_X\big(m(K_X+\Delta)\big)$ with semipositive curvature current, such that the multiplier ideal sheaf $\mathcal I(h)$ is different from $\OO_Y$. Then $\kappa(X,K_X+\Delta)\geq0$.
\end{cor}

\begin{proof} 
By passing a to $\Q$-factorialisation, we may assume that $X$ is $\Q$-factorial. By Theorem \ref{thm:multiplier2} applied to $L=m(K_X+\Delta)$, there exists an effective $\Q$-divisor $D$ on $X$ such that $K_X+\Delta\equiv D$ and $\kappa(X,D)\geq0$. By \cite[Theorem 0.1]{CKP12} we have $\kappa(X,K_X+\Delta)\geq\kappa(X,D)$, and the result follows.
\hfill $\Box$
\end{proof}

The following is the main result of this section.

\begin{thm}\label{thm:nu1}
Let $X$ be a minimal  projective terminal variety such that $\nu(X,K_X)=1$. If $\chi(X,\OO_X)\neq0$, then $\kappa(X,K_X)\geq0$.
\end{thm}

\begin{proof}
If there exist a resolution $\pi\colon Y\to X$, a positive integer $m$ such that $mK_X$ is Cartier, and a singular metric $h$ on $\pi^*\OO_X(mK_X)$ with semipositive curvature current, such that the multiplier ideal sheaf $\mathcal I(h)$ is different from $\OO_Y$, then the result follows from Corollary \ref{cor:multiplier}. 

Otherwise, pick a resolution $\pi\colon Y\to X$, a positive integer $m$ such that $mK_X$ is Cartier, and a singular metric $h$ on $\pi^*\OO_X(mK_X)$ with semipositive curvature current. Then the result follows from Corollary \ref{cor:nd1}; note that since $\mathcal I(h^{\otimes\ell})=\OO_Y$ for all $\ell$, here the hypothesis in Corollary  \ref{cor:nd1} that $h$ has algebraic singularities is not necessary.
\hfill $\Box$
\end{proof}

\begin{rem}\label{remark:dim3}{\rm
Theorem \ref{thm:nu1} gives a new proof of the hardest part of nonvanishing for minimal terminal threefolds. Indeed, if $X$ is a minimal terminal threefold, we only need to check the cases $\nu(X,K_X)\in\{1,2\}$. If the irregularity $q(X)$ is positive, then the nonvanishing follows from known cases of Iitaka's conjecture $C_{n,m}$ applied to the Albanese map, see for instance \cite[pp.\ 73-74]{MP97}. The case $\nu(X,K_X)=2$ is a relatively quick application of Miyaoka's inequality for Chern classes and the Kawamata-Viehweg vanishing, see \cite[pp.\ 83-84]{MP97}. The remaining case $\nu(X,K_X)=1$ is the most difficult. Since we may assume that $q(X) = 0,$ we are reduced to the case that $\chi(X,\OO_X)>0$. Then the nonvanishing follows from Theorem \ref{thm:nu1}.}
\end{rem}

\begin{thm} \label{thm:nonvanishingFormsnu2}
Let $X$ be a minimal  projective terminal pair of dimension $n$ such that $\nu(X,K_X)=1$. Let $\pi\colon Y\to X$ be a resolution of $X$, fix a tensor representation $\mathcal E$ of $\Omega_Y^1$, and let $q$ be the rank of $\mathcal E$. Suppose that 
$$ h^0\big(Y,\mathcal E\otimes \OO_Y(m_0\pi^*K_X)\big) \geq  q+1$$
for some $m_0$ such that $m_0K_X$ is Cartier. Then $K_X$ is semiample.
\end{thm} 

\begin{proof}
We follow the proof of Theorem \ref{thm:abundanceForms} closely. Let $\mathcal F$ be the subsheaf of $\mathcal E \otimes \OO_Y(m_0\pi^*K_X)$ generated by its global sections, and let $r$ be the rank of $\mathcal F$. Then
$$\det\mathcal F\subseteq\bigwedge^r\mathcal E \otimes \OO_Y(rm_0\pi^*K_X),$$
and there exists a Cartier divisor $N$ such that  $\OO_Y(N)$ is the saturation of $\det\mathcal F$ in $\bigwedge^r\mathcal E \otimes \OO_Y(rm_0\pi^*K_X)$. By Proposition \ref{pro:wedge} we have
\begin{equation}\label{eq:infmany2a}
h^0(Y,N)\geq2,
\end{equation}
and denote 
\begin{equation}\label{eq:18a}
{-}F_Y=N-rm_0\pi^*K_X.
\end{equation}
Then we have the exact sequence
$$ 0 \to \OO_Y(-F_Y) \to \bigwedge^r\mathcal E \to \mathcal Q \to 0.$$
Since $\OO_Y(-F_Y)$ is saturated in $\bigwedge^r\mathcal E$, the sheaf $\mathcal Q$ is torsion-free, and hence $F=c_1(\mathcal Q)$ is pseudoeffective by Theorem \ref{thm:CP11}. From the above exact sequence, there exists a positive integer $\ell$ such that $\ell K_Y\sim F-F_Y$, and \eqref{eq:18a} gives
$$N+F\sim \ell K_Y+rm_0\pi^*K_X.$$
Pushing forward this relation by $\pi$, we obtain
$$\pi_*N+\pi_*F\sim_\Q (rm_0+\ell)K_X.$$
Since $\pi_*F$ is pseudoeffective and $\kappa(X,\pi_*N)\geq1$ by Lemma \ref{lem:pushforward}, Corollary \ref{cor:nu1} implies that $\kappa(X,K_X)\geq1$, and hence $\kappa(X,K_X)=\nu(X,K_X)=1$. Now  conclude by Lemma \ref{lem:Kappa=KappaSigma}.
\hfill $\Box$
\end{proof}

\begin{cor} \label{cor:xx}
Let $X$ be a minimal  projective terminal variety of dimension $n$ such that $\nu(X,K_X)=1$. 
\begin{enumerate}
\item[\rm (a)] If $\vert\chi(X,\OO_X)\vert >  2^{n-1}$, then $K_X$ is semiample. 
\item[\rm (b)] If $\kappa(X,K_X)\geq0$, $\pi_1(X)$ is infinite and $\chi(X,\OO_X)\neq 0$, then $K_X$ is semiample. 
\end{enumerate}
\end{cor}

\begin{proof} 
For (a), let $\pi\colon Y\to X$ be a resolution. If $\vert\chi(X,\OO_X)\vert =\vert\chi(Y,\OO_Y)\vert>  2^{n-1}$, then there exists $q$ such that 
$$ h^0\big(Y,\Omega^q_Y) =h^q(Y,\OO_Y) > \binom{n}{q}.$$ 
The result follows by Theorem \ref{thm:nonvanishingFormsnu2}.

For (b), by Lemma \ref{lem:Kappa=KappaSigma} it suffices to show that $\kappa(X,K_X)=\nu(X,K_X)$. Arguing by contradiction, assume that $\kappa(X,K_X)=0$. Let $\pi\colon Y\to X$ be a resolution of $X$. Note that $\vert\chi(X,\OO_X)\vert=\vert\chi(X,\OO_X)\vert\neq 0$ since $X$ has rational singularities, and that $\pi_1(Y)$ is infinite by \cite[Theorem 1.1]{Tak03}. In order to derive contradiction, by \cite[Corollary 5.3]{Cam95} it suffices to show that for any $q$ and for any coherent subsheaf $\mathcal F\subseteq\Omega_Y^q$ we have $\kappa(Y,\det\mathcal F)\leq0$.

To this end, we follow the proof of Theorem \ref{thm:nonvanishingFormsnu2}. Fix $q$, and assume that there exists a coherent subsheaf $\mathcal F\subseteq\Omega_Y^q$ of rank $r$ such that $\kappa(Y,\det\mathcal F)>0$. Then
$$\det\mathcal F\subseteq\bigwedge^r\Omega_Y^q,$$
and there exists a Cartier divisor $N\geq0$ such that  $\OO_Y(N)$ is the saturation of $\det\mathcal F$ in $\bigwedge^r\Omega_Y^q$. Then we have the exact sequence
$$ 0 \to \OO_Y(N) \to \bigwedge^r\Omega_Y^q \to \mathcal Q \to 0,$$
where the sheaf $\mathcal Q$ is torsion-free, and hence $F=c_1(\mathcal Q)$ is pseudoeffective by Theorem \ref{thm:CP11}. From the above exact sequence, there exists a positive integer $\ell$ such that $N+F\sim \ell K_Y$. Pushing forward this relation by $\pi$, we obtain
$$\pi_*N+\pi_*F\sim_\Q \ell K_X,$$
hence Corollary \ref{cor:nu1} and Lemma \ref{lem:pushforward} imply 
$$0<\kappa(Y,\det\mathcal F)\leq\kappa(Y,N)\leq\kappa(X,\pi_*N)\leq\kappa(X,K_X)=0,$$
a contradiction which finishes the proof.
\hfill $\Box$
\end{proof}

We also notice:

\begin{thm} 
Let $X$ be a minimal  projective klt variety of dimension $n$ such that $\nu(X,K_X) = 1$. If for some $q>0$ there exists 
$$ s \in H^0\Big(X, \Omega^{[q]}_X\Big)$$
which vanishes along some divisor, then $\kappa (X,K_X) \geq 0.$ 
\end{thm} 

\begin{proof}  
Suppose that $s$ vanishes along a divisor $D\geq0$, and let $\pi\colon Y\to X$ be a resolution. Then by \cite[Theorem 4.3]{GKKP11}, the pullback $s_Y\in H^0(Y,\Omega_Y^q)$ of $s$ exists, and vanishes along the divisor $D_Y=\pi^{-1}_*D$. There is a Cartier divisor $D'\geq D_Y$ such that $\OO_Y(D')$ is the saturation of $\OO_Y(D_Y)$ in $\Omega^q_Y$, hence we have an exact sequence 
$$ 0 \to \OO_Y(D') \to  \Omega^q_Y  \to \mathcal Q \to  0,$$
where $\mathcal Q$ is torsion-free. The divisor $c_1(\mathcal Q)$ is pseudoeffective by Theorem \ref{thm:CP11}, and there exists a positive integer $\ell$ such that 
$$\ell K_Y \sim_\Q  D' + c_1(\mathcal Q).$$
By Theorem \ref{thm:nu1a}, this implies $\kappa(X,K_X)\geq\kappa(Y,K_Y)\geq\kappa(Y,D')\geq0$, which finishes the proof. 
\hfill $\Box$
\end{proof} 

\section{Uniruled varieties}\label{sec:uniruled}

When the underlying variety of a klt pair $(X,\Delta)$ is uniruled, which -- as explained in the introduction -- is equivalent to the canonical class not being pseudoeffective on a terminalisation of $X$, we can show that the pair in many circumstances has a good model. The following is Theorem \ref{thm:E}.

\begin{thm}\label{thm:uniruled}
Assume the existence of good models for klt pairs in dimensions at most $n-1$. Let $(X,\Delta)$ be a klt pair of dimension $n$ such that $X$ is uniruled and $K_X+\Delta$ is pseudoeffective. 
If  $\vert\chi(X,\OO_X)\vert >  2^{n-1}$,  then 
$(X,\Delta)$ has a good model.
\end{thm}

\begin{proof}
We follow closely the proof of \cite[Theorem 1.3]{DL15}. First of all, by passing to a resolution and by \cite[Lemma 2.9]{DL15}, we may assume that the pair $(X,\Delta)$ is log smooth, the divisor $\Delta$ is reduced, and there exists a $\Q$-divisor $D\geq0$ such that $K_X+\Delta\sim_\Q D$ and $\Supp\Delta=\Supp D$. Then the proof of \cite[Theorem 3.5]{DL15} shows that there are proper maps
$$T\stackrel{\mu}{\lto} W\stackrel{g}{\lto} X'\stackrel{\pi}{\lto} X,$$
where $\pi$ and $\mu$ are finite and $g$ is birational, and a $\Q$-divisor $\Delta_T$ on $T$ such that $(T,\Delta_T)$ is a log smooth klt pair with $|K_T|\neq\emptyset$ and
$$\kappa(T,K_T+\Delta_T)=\kappa(X,K_X+\Delta)\quad\text{and}\quad\nu(T,K_T+\Delta_T)=\nu(X,K_X+\Delta).$$ 
We argue by contradiction and assume that $(X,\Delta)$ does not have a good model. Then the pair $(T,\Delta_T)$ does not have a good model by Lemma \ref{lem:Kappa=KappaSigma}, hence
\begin{equation}\label{eq:uniruled}
h^q(T,\OO_T)\leq\binom{n}{q}\quad\text{for all }q
\end{equation}
by Corollary \ref{cor:chi}. Since $\OO_W$ is a direct summand of $\mu_*\OO_T$ and since $\OO_X$ is a direct summand of $\pi_*\OO_{X'}$ by \cite[Proposition 5.7]{KM98}, by the Leray spectral sequence
we have for each $q$ 
\begin{equation}\label{eq:uniruled1}
h^q(T,\OO_T)\geq h^q(W,\OO_W)\quad\text{and}\quad h^q(X',\OO_{X'})\geq h^q(X,\OO_X).
\end{equation}
We claim that $X'$ is a klt variety. The claim immediately implies the theorem: indeed, since then $X'$ has rational singularities, we have $h^q(W,\OO_W)=h^q(X',\OO_{X'})$, which together with \eqref{eq:uniruled} and \eqref{eq:uniruled1} gives 
$$h^q(X,\OO_X)\leq\binom{n}{q}\quad\text{for all }q,$$
and hence $\chi(X,\OO_X)\leq2^{n-1}$. 

To show the claim, by the proof of \cite[Theorem 3.5]{DL15} there exists a $\Q$-divisor $\Delta'$ on $X'$ such that $K_{X'}+\Delta'=\pi^*(K_X+\Delta)$ and $\Supp\Delta'=\Supp\pi^*\Delta$. Then for a rational number $0<\varepsilon\ll1$, denoting $\Delta''=\Delta'-\varepsilon\pi^*\Delta$, we have $\Delta''\geq0$ and
$$K_{X'}+\Delta''=\pi^*(K_X+(1-\varepsilon)\Delta).$$
Since the pair $(X,(1-\varepsilon)\Delta)$ is klt, so is the pair $(X',\Delta'')$ by \cite[Proposition 5.20]{KM98}, hence so is $X'$, which finishes the proof.
\hfill $\Box$
\end{proof}

\section{Nonvanishing on Calabi-Yau varieties}\label{sec:CY}

As mentioned in the introduction, similar techniques as those used above can be applied in the context of nef line bundles on varieties of Calabi-Yau type. In particular, Theorem \ref{thm:FormsCY0} generalises \cite[Proposition 3.4]{LOP16} and \cite[3.1]{Wi94} to any dimension. An immediate corollary is Theorem \ref{thm:G}.

\begin{thm} \label{thm:FormsCY0}
Assume the existence of good models for klt pairs in dimensions at most $n-1$. Let $X$ be a projective klt variety of dimension $n$ such that $K_X\sim_\Q0$, and let $L$ be a nef Cartier divisor on $X$.   Let $\pi\colon Y\to X$ be a resolution of $X$. Assume that for some positive integer $p$ we have
\begin{equation}\label{eq:eqeq}
H^0\big(Y,(\Omega^1_Y)^{\otimes p} \otimes \OO_Y(m\pi^*L)\big) \neq 0
\end{equation}
for infinitely many $m$. Then $\kappa(X,L) \geq 0$. 
\end{thm} 

\begin{proof} 
First assume that $L \equiv 0 $. If $ H^1(X,\OO_X) = 0$, then $L$ is torsion, i.e.\ there is a positive integer $m$ such that $mL \sim 0$. Hence, we may assume that $H^1(X,\OO_X) \ne 0$. 
Then $X$ has a non-trivial Albanese variety $A$ and there exist a finite quasi-\'etale cover $\eta\colon \widetilde X \to X$ and a variety $Z$ with canonical singularities such that $H^1(Z,\OO_Z)=0$ and 
$$ \widetilde X \simeq Z \times A,$$
see \cite[Corollary 3.6]{GKP16}, \cite[Proposition 9.1]{Dru17}. By further replacing $\widetilde X$ by a suitable finite \'etale cover, we may assume that $\OO_Z(\eta^*L) \simeq \OO_Z$. Now, consider a desingularisation $\widetilde Z \to Z$ and set $\widetilde Y = \widetilde Z \times A$, so that we obtain a desingularisation $\widetilde \pi\colon \widetilde Y \to \widetilde X$. By possibly blowing up further, we may assume that $\widetilde\pi$ factors through the natural projection $\widetilde X \times_X Y\to\widetilde X$.

Then by \cite[Exercise III.12.6]{Har77} there exists a line bundle $\mathcal M$ on $A$ such that
$$ \OO_{\widetilde Y}(\widetilde\pi^*\eta^*L)  \simeq p_A^*\mathcal M,$$
where $p_A\colon \widetilde Y\to A$ is the second projection, and note that $\mathcal M\equiv0$. Furthermore, \eqref{eq:eqeq} gives
\begin{equation}\label{eq:eqeqeq}
H^0\big(\widetilde Y, (\Omega^1_{\widetilde Y})^{\otimes p} \otimes p_A^*\mathcal M^{\otimes m}\big)\simeq H^0\big(\widetilde Y, (\Omega^1_{\widetilde Y})^{\otimes p} \otimes \OO_{\widetilde Y}(m\widetilde\pi^*\eta^* L)\big)\neq0
\end{equation}
for infinitely many $m$. If $\mathcal M$ is torsion, then $L$ is also torsion. Otherwise, let $p_{\widetilde Z}\colon \widetilde Y \to \widetilde Z$ denote the first projection, and for every positive integer $m$ consider the sheaf 
$$\mathcal F_m=p_{\widetilde Z,*}\big((\Omega^1_{\widetilde Y})^{\otimes p} \otimes p_A^*\mathcal M^{\otimes m}\big).$$
Then, since every fibre of $p_{\widetilde Z}$ is isomorphic to $A$, we obtain $H^0(\widetilde Z,\mathcal F_m)=0$, which contradicts \eqref{eq:eqeqeq}.

From now on we assume that $L \not \equiv 0$ and  follow closely the proof of Theorem \ref{thm:nonvanishingForms}. We may assume that $X$ is $\Q$-factorial. Then there is an infinite subset $\mathcal T\subseteq \N$ such that \eqref{eq:eqeq} holds for all $m\in\mathcal T$. 
We then apply Lemma \ref{lemmafund} to $\mathcal E = (\Omega^1_Y)^{\otimes p} $ and $\mathcal L = \OO_Y(\pi^*L)$, and then Proposition \ref{pro:quotient} as in Theorem \ref{thm:nonvanishingForms} 
to obtain a positive integer $\ell$, a pseudoeffective divisor $F$ and integral divisors $N_m\geq0$ such that 
$$N_m+F\sim m\pi^*L+\ell K_Y$$
for infinitely many positive integers $m$. Pushing forward this relation to $X$, we get
$$\pi_*N_m+\pi_*F\sim_\Q mL,$$
and conclude by Theorem \ref{thm:MMPtwistCY}.
\hfill $\Box$
\end{proof} 

\begin{thm} \label{thm:MMPtwistCY}
Assume the existence of good models for klt pairs in dimensions at most $n-1$. Let $X$ be a $\Q$-factorial projective klt variety of dimension $n$ such that $K_X\sim_\Q0$, and let $L$ be a nef divisor on $X$. Assume that there exist a pseudoeffective $\Q$-divisor $F$ on $X$ and an infinite subset $\mathcal S\subseteq\N$ such that 
\begin{equation}\label{eq:rel2aCY}
N_m+F\sim_\Q mL,
\end{equation}
for all $m\in\mathcal S$, where $N_m\geq0$ are integral Weil divisors. Then 
$$\kappa(X,L)=\max\{\kappa(X,N_m)\mid m\in\mathcal S\}\geq0.$$
\end{thm} 

\begin{proof}
The proof is very similar to that of Theorem \ref{thm:MMPtwist}.

\medskip

\emph{Step 1.}
Note first that \eqref{eq:rel2aCY} implies
\begin{equation}\label{eq:rel2bCY}
N_p-N_q\sim_\Q (p-q)L\quad\text{for all }p,q\in\mathcal S.
\end{equation}

There are three cases to consider. First assume that $N_p$ is big for some $p\in\mathcal S$. Then \eqref{eq:rel2aCY} implies that $L$ is big, and the result is clear. 

\medskip

\emph{Step 2.}
Now assume that $\kappa(X,N_p)=\kappa(X,N_q)=0$ for some distinct $p,q\in\mathcal S$. Let $r\in\mathcal S$ be such that $r>q$. Then by \eqref{eq:rel2bCY} we have
$$L\sim_\Q\frac{1}{q-p}(N_q-N_p)\quad\text{and}\quad L\sim_\Q \frac{1}{r-p}(N_r-N_p),$$
so that
$$(r-p)N_q\sim_\Q(q-p)N_r+(r-q)N_p\geq0.$$
Since $\kappa(X,N_q)=0$, we have $(r-p)N_q=(q-p)N_r+(r-q)N_p$, and hence $\Supp N_r\subseteq\Supp N_q$ and $\kappa(X,N_r)=0$. Therefore, for $r>q$, all divisors $N_r$ are supported on a reduced Weil divisor. By Lemma \ref{relation}, there are positive integers $k\neq\ell$ larger than $q$ in $\mathcal S$ such that $N_k\leq N_\ell$, and hence by \eqref{eq:rel2bCY},
$$(\ell-k)L\sim_\Q N_\ell-N_k\geq0,$$
which shows that $\kappa(X,L)\geq0$. Moreover, since then $\kappa(X,L)\leq\kappa(X,N_q)=0$ by \eqref{eq:rel2bCY}, we have
$$\kappa(X,L)=0.$$
If $m$ is an element of $\mathcal S$ with $m\geq q$, then $\kappa(X,N_m)=0$ by above, and if $m<q$, then $0=\kappa(X,N_q)\geq\kappa(X,N_m)$ by \eqref{eq:rel2bCY}, which yields the result.

\medskip

\emph{Step 3.}
Finally, we may assume that
\begin{equation}\label{eq:kodaira1CY}
0<\kappa(X,N_p)<n \quad\text{for every }p\in\mathcal S.
\end{equation}
Fix integers $\ell>k$ in $\mathcal S$ and fix $0<\varepsilon,\delta\ll1$ such that:
\begin{enumerate}
\item[(a)] the pair $(X,\varepsilon N_k)$ is klt, 
\item[(b)] $\varepsilon(\ell-k)>2n$, and
\item[(c)]  the pair $(X,\delta N_\ell)$ is klt.
\end{enumerate}
Fix an ample divisor $A$ on $X$, and we run the MMP with scaling of $A$ for the klt pair $(X,\delta N_\ell)$. Since we are assuming the existence of good models for klt pairs in lower dimensions, by Theorem \ref{thm:lai} our MMP with scaling of $A$ terminates with a good model for $(X,\delta N_\ell)$. 

We claim that this MMP is $L$-trivial, and hence the proper transform of $L$ at every step of this MMP is a nef Cartier divisor. Indeed, it is enough to show the claim for the first step of the MMP, as the rest is analogous. Let $c_R\colon X\to Z$ be the contraction of a $(K_X+\delta N_\ell)$-negative (hence $N_\ell$-negative) extremal ray $R$ in this MMP. Since by \eqref{eq:rel2bCY} we have
$$N_\ell\sim_\Q N_k+(\ell-k)L$$
and as $L$ is nef, the ray $R$ is also $N_k$-negative. By the boundedness of extremal rays \cite[Theorem 1]{Kaw91}, there exists a rational curve $C$ contracted by $c_R$ such that $\varepsilon N_k\cdot C=(K_X+\varepsilon N_k)\cdot C\geq {-}2n$. If $c_R$ were not $L$-trivial, then $L\cdot C\geq1$ as $L$ is Cartier. But then the condition (b) above yields
$$\varepsilon N_\ell\cdot C=\varepsilon N_k\cdot C+\varepsilon(\ell-k)L\cdot C>0,$$
a contradiction which proves the claim, i.e.\ the MMP is $L$-trivial.

\medskip

\emph{Step 4.}
In particular, the numerical Kodaira dimension and the Kodaira dimension of $L$ are preserved in the MMP, see \cite[Theorem 3.7(4)]{KM98} and \S\ref{subsec:numdim}. Hence, $L$ is not big by \eqref{eq:rel2bCY} and \eqref{eq:kodaira1CY}. Furthermore, the proper transform of $F$ is pseudoeffective. Therefore, by replacing $X$ by the resulting minimal model, we may assume that $N_\ell$ is semiample. Note also that $\kappa(X,N_m)>0$ for all $m\in\mathcal S$ by Lemma \ref{lem:pushforward}.

Fix $m\in\mathcal S$ such that $m>\ell$. Then the divisor 
$$N_m\sim_\Q N_\ell+(m-\ell)L$$
is nef. Notice that $N_m$ is not big, since otherwise $L$ would be big as in Step 1. Therefore, we have $0<\kappa(X,N_m)<n$, and pick $0<\eta\ll1$ so that the pair $(X,\eta N_m)$ is klt.  Since we are assuming the existence of good models for klt pairs in lower dimensions, by Theorem \ref{thm:lai} the divisor $\eta N_m$ is semiample. 

Let $\varphi_\ell\colon X\to S_\ell$ and $\varphi_m\colon X\to S_m$ be the Iitaka fibrations associated to $N_\ell$ and $N_m$, respectively. Then there exist ample $\Q$-divisors $A_\ell$ on $S_\ell$ and $A_m$ on $S_m$ such that 
$$N_\ell\sim_\Q\varphi_\ell^*A_\ell\quad\text{and}\quad N_m\sim_\Q\varphi_m^*A_m.$$ 
If $\xi$ is a curve on $X$ contracted by $\varphi_m$, then by \eqref{eq:rel2bCY} we have
$$0=N_m\cdot \xi=N_\ell\cdot \xi+(m-\ell)L\cdot \xi,$$
hence $N_\ell\cdot \xi=L\cdot \xi=0$. In particular, $\xi$ is contracted by $\varphi_\ell$, which implies that there exists a morphism $\psi\colon S_m\to S_\ell$ such that $\varphi_\ell=\psi\circ\varphi_m$. Therefore, denoting $B=\frac{1}{m-\ell} (A_m-\psi^*A_\ell)$, we have
$$L\sim_\Q \frac{1}{m-\ell}(N_m-N_\ell)\sim_\Q\varphi_m^*B.$$
Denoting $B_0=mB-A_m$, it is easy to check from \eqref{eq:rel2aCY} that 
$$F\sim_\Q\varphi_m^*B_0,$$
and hence $B_0$ is pseudoeffective. Therefore the divisor $A_m+B_0$ is big on $S_m$, and
$$mL\sim_\Q N_m+F\sim_\Q\varphi_m^*(A_m+B_0).$$
This yields
\begin{equation}\label{eq:equality}
\kappa(X,L)=\kappa(S_m,A_m+B_0)=\dim S_m=\kappa(X,N_m),
\end{equation}
and note that this holds for $m\in\mathcal S$ sufficiently large. In particular, $\kappa(X,L)\geq0$, and then $\kappa(X,L)\geq\kappa(X,N_p)$ for all $p\in\mathcal S$ by \eqref{eq:equality} and \eqref{eq:rel2bCY}, which finishes the proof.
\hfill $\Box$
\end{proof}

\begin{thm} \label{thm:FormsCY}
Assume the existence of good models for klt pairs in dimensions at most $n-1$. Let $X$ be a  projective klt variety of dimension $n$ such that $K_X\sim_\Q0$, and let $L$ be a nef Cartier divisor on $X$ such that $\kappa(X,L)\geq0$. 
 Let $\pi\colon Y\to X$ be a resolution of $X$. Fix a tensor representation $\mathcal E$ of $\Omega_Y^1$, and let $q$ be the rank of $\mathcal E$. Assume that there is a positive integer $m$  such that 
$$ h^0\big(Y,\mathcal E\otimes \OO_Y(m_0\pi^*L)\big) \geq q+1.$$
Then $L$ is semiample. 
\end{thm} 

\begin{proof} 
We follow closely the proof of Theorem \ref{thm:abundanceForms}. 
Then, analogously as in that proof, there exist a Cartier divisor $N$ and an integer $1\leq r\leq n$ such that $\OO_Y(N)$ is a saturated subsheaf of $\bigwedge^r\mathcal E\otimes\OO_Y(rm_0\pi^*L)$ and
\begin{equation}\label{eq:infmany2CY}
h^0(Y,N)\geq2.
\end{equation}
Denote 
\begin{equation}\label{eq:18CY}
-F_Y=N-rm_0\pi^*L.
\end{equation}
Then we have the exact sequence
$$ 0 \to \OO_Y(-F_Y) \to \bigwedge^r\mathcal E \to \mathcal Q \to 0.$$
Since $\OO_Y(-F_Y)$ is saturated in $\bigwedge^r\mathcal E$, the sheaf $\mathcal Q$ is torsion-free, and hence $\widetilde F_Y=c_1(\mathcal Q)$ is pseudoeffective by Theorem \ref{thm:CP11}. From the above exact sequence, there exists a positive integer $\ell$ such that $\ell K_Y\sim\widetilde F_Y-F_Y$. 

Let $d$ be the smallest positive integer such that $H^0(Y,d\pi^*L)\neq0$. Denote $\mathcal S=\{rm_0+id\mid i\in\N\}\subseteq\Z$ and $\widetilde N_m=N+(m-rm_0)\pi^*L$ for $m\in\mathcal S$. From \eqref{eq:infmany2CY}, for every $m\in\mathcal S$ we have $\kappa(Y,\widetilde N_m)\geq1$, and \eqref{eq:18CY} gives
\begin{equation}\label{eq:relation2CY}
\widetilde N_m+\widetilde F_Y\sim m\pi^*L+\ell K_Y.
\end{equation}
Denote $N_m=\pi_*\widetilde N_m$ and $F=\pi_*\widetilde F_Y$. Then $F$ is pseudoeffective, and by Lemma \ref{lem:pushforward} we have $\kappa(X,N_m)\geq\kappa(Y,\widetilde N_m)\geq1$. Pushing forward the relation \eqref{eq:relation2CY} to $X$, we get
$$N_m+F\sim_\Q mL.$$
As in Steps 1, 3 and 4 of the proof of Theorem \ref{thm:MMPtwistCY}, we conclude that 
$$\kappa(X,L)=\max\{\kappa(X,N_m)\mid m\in\mathcal S\}\geq1.$$
Pick a rational number $0<\varepsilon\ll1$ such that the pair $(X,\varepsilon L)$ is klt. Then $\varepsilon L\sim_\Q K_X+\varepsilon L$ is semiample by Theorem \ref{thm:lai}, a contradiction which finishes the proof.
\hfill $\Box$
\end{proof} 

\begin{rem}{\rm
Let $X$ be a projective canonical variety of dimension $n$ such that $K_X\sim_\Q0$. Then we have
$$ h^q(X,\OO_X) = h^0\big(X,\Omega^{[q]}_X\big) \leq \binom{n}{q}\quad\text{for all }q,$$ 
and in particular, $|\chi(X,\mathcal O_X)| \leq 2^{n-1}$. 
Indeed, assume that there exists $q$ such that $h^q(X,\OO_X) > \binom{n}{q}$, and note that $h^q(X,\OO_X) = h^0\big(X,\Omega^{[q]}_X\big)$ by \cite[Proposition 6.9]{GKP16}. Then by Proposition \ref{pro:wedge} there is a positive integer $N$ and a line bundle $\mathcal L \subseteq  \big(\bigwedge^N\Omega^{[q]}_X\big)^{**}$ with $h^0(X,\mathcal L) \geq 2$. Let $C\subseteq X$ be a curve obtained as complete intersection of $n-1$ high multiples of a very ample divisor. By Miyaoka's generic semipositivity \cite{Miy87,Miy87a}, the sheaf $\big(\bigwedge^N\Omega^{[q]}_X\big)^{**}|_C$ is nef, and hence semistable with respect to any polarisation since $\det(\bigwedge^N\Omega^{[q]}_X\big)^{**}=\OO_X$. Therefore, $\big(\bigwedge^N\Omega^{[q]}_X\big)^{**}$ is semistable with respect to any polarisation by the theorem of Mehta-Ramanathan \cite{MR82,Fle84}. However, the slope of $\mathcal L$ with respect to any ample polarisation is positive, a contradiction.

In the more general setting when $X$ has klt singularities, the techniques from Sections \ref{sec:thmC} and \ref{sec:uniruled} show that $|\chi(X,\mathcal O_X)| \leq 2^{n-1}$ assuming the Minimal Model Program in dimensions at most $n-1$.}
\end{rem}

\begin{cor}\label{cor:nef}
Assume the existence of good models for klt pairs in dimensions at most $n-1$. Let $X$ be a projective klt variety of dimension $n$ such that $K_X\sim_\Q0$, and let $\mathcal L$ be a nef line bundle on $X$. 
\begin{enumerate}
\item[\rm (i)] Assume that $\mathcal L$ has a singular hermitian metric with semipositive curvature current and with algebraic singularities. If $\chi(X,\OO_X)\neq0$, then $\kappa(X,\mathcal L)\geq0$.
\item[\rm (ii)] If $\mathcal L$ is hermitian semipositive and if $\chi(X,\OO_X)\neq0$, then $\mathcal L$ is semiample.
\end{enumerate} 
\end{cor}

\begin{proof}
The proof is similar to that of Corollaries \ref{cor:nv}, \ref{cor:semipositive} and \ref{cor:chi}, so we will be quick on the details. By passing to a $\Q$-factorialisation and replacing $\mathcal L$ by its pullback, we may assume that $X$ is $\Q$-factorial.

\medskip

For (i), choose a resolution $\pi\colon Y\to X$ as in \S\ref{subsec:metric}, and let $h$ denote the induced metric on $\pi^*\mathcal L$. Then the local plurisubharmonic weights $\phi$ of $h$ are of the form
$$\phi = \sum_{j=1}^r \lambda_j \log \vert g_j \vert + O(1),$$
where $\lambda_j$ are positive rational numbers and the divisors $D_j$ defined locally by $g_j$ form a simple normal crossing divisor on $Y$. We have
\begin{equation}\label{eq:metric11}
\textstyle\mathcal I(h^{\otimes m})=\OO_Y\big(-\sum_{j=1}^r\lfloor m\lambda_j\rfloor D_j\big).
\end{equation}
Assume that $\kappa(X,\mathcal L)=-\infty$. Then by Theorem \ref{thm:FormsCY0}, for all $p\geq 0$ and for all $m$ sufficiently divisible we have
$$ H^0(Y,\Omega^p_Y \otimes \pi^*\mathcal L^{\otimes m})=0,$$
and thus
$$H^0(Y,\Omega^p_Y\otimes \pi^*\mathcal L^{\otimes m}\otimes\mathcal I(h^{\otimes m})) = 0.$$
Theorem \ref{thm:DPS} implies that for all $p\geq 0$ and for all $m$ sufficiently divisible,
$$H^p(Y,\OO_Y(K_Y)\otimes \pi^*\mathcal L^{\otimes m}\otimes\mathcal I(h^{\otimes m})) = 0,$$
which together with \eqref{eq:metric11} and Serre duality yields 
$$\textstyle\chi\big(Y,\OO_Y\big(\sum_{j=1}^r\lfloor m\lambda_j\rfloor D_j\big)\otimes \pi^*\mathcal L^{\otimes {-}m}\big) = 0$$
for all $m$ sufficiently divisible. As in the proof of Corollary \ref{cor:nv}, this then implies $\chi(X,\OO_X) = 0$, which proves (i).

\medskip

Now assume that $\mathcal L$ is hermitian semipositive and that $\chi(X,\OO_X)\neq0$. By (i) we have $\kappa(X,\mathcal L)\geq0$, and choose a Cartier divisor $L\geq0$ such that $\mathcal L\simeq\OO_X(L)$. Pick a rational number $0<\varepsilon\ll1$ such that the pair $(X,\varepsilon L)$ is klt. Then \cite[Theorem 5.1]{GM17} implies that $\varepsilon L\sim_\Q K_X+\varepsilon L$ is semiample, which proves (ii).
\hfill $\Box$
\end{proof}

\begin{rem}{\rm
Let $X$ be a projective klt variety of dimension $n$ such that $K_X\sim_\Q0$. If we drop the assumption $\chi(X,\mathcal O_X) \ne 0$, then Corollary \ref{cor:nef} fails in general: for instance, $X$ could be a torus. However, one might expect that there is a line bundle $\mathcal L'$ numerically equivalent to $\mathcal L$ such that the conclusion remains true. Furthermore, note that if $K_X \sim 0$ and if $n $ is odd, then we always have $\chi(X,\mathcal O_X) = 0$ by \cite[Corollary 6.11]{GKP16}. }
\end{rem} 

Recall that a Calabi-Yau manifold of dimension $n$ is a  simply connected projective manifold $X$ such that $K_X\sim 0$ and $h^q(X,\OO_X)=0$ for all $q=1,\dots,n-1$. By the Beauville-Bogomolov decomposition theorem, Calabi-Yau manifolds are the building blocks for all manifolds with $K_X\equiv0$, together with hyperk\"ahler manifolds and abelian varieties. 

\begin{thm}\label{thm:CYnu1} 
Let $X$ be a projective manifold with $K_X \sim_\Q 0$ and $\pi_1(X) $ finite. Let $\mathcal L$ be a nef line bundle on $X$ with $\nu(X,\mathcal L)=1$. Let $\eta\colon \widetilde X \to X$ be the universal cover and assume that the Beauville-Bogomolov decomposition is of the form
$$\widetilde X \simeq \prod X_j,$$ 
where all irreducible components $X_j$ are even-dimensional. Then $\kappa(X,\mathcal L)\geq0$.
\end{thm}

\begin{proof} 
Replacing $X$ by $\widetilde X$ and observing that $\kappa(X,\mathcal L) = \kappa(\widetilde X,\eta^*\mathcal L)$ and $\nu(X,\mathcal L) = \nu(\widetilde X,\eta^*\mathcal L)$, we may assume that $X$ is simply connected. Since all $X_j$ are even-dimensional Calabi-Yau or hyperk\"ahler manifolds, we have $\chi(X_j,\OO_{X_j}) > 0$ for all $j$, hence $\chi(X,\OO_X) > 0$. Since $h^1(X,\OO_X)=0$, numerical and linear equivalence of divisors on $X$ coincide.

If there exist a positive integer $m$ and a singular metric $h$ on $\mathcal L^{\otimes m}$ with semipositive curvature current, such that the multiplier ideal sheaf $\mathcal I(h)$ is different from $\OO_X$, then the result follows from Corollary \ref{cor:multiplier}. 

Otherwise, pick  a singular metric $h$ on $\mathcal L$ with semipositive curvature current. Then the result follows from the proofs of Corollary \ref{cor:nef}(i) and Theorem \ref{thm:FormsCY0}, by invoking Theorem \ref{thm:nu1a} instead of Theorem \ref{thm:MMPtwistCY} in the proof of Theorem \ref{thm:FormsCY0}. Note that since $\mathcal I(h^{\otimes\ell})=\OO_Y$ for all $\ell$, the hypothesis in Corollary \ref{cor:nef}(i) that $h$ has algebraic singularities is not necessary.
\hfill $\Box$
\end{proof}

\begin{rem}{\rm
Note that if in the theorem above $X_j$ is a hyperk\"ahler manifold of dimension $\geq4$ for some $j$, then $\eta^*\mathcal L|_{X_j}\simeq\OO_{X_j}$ by \cite[Lemma 1]{Mat99}.}
\end{rem}

We also have the following generalization of Theorem \ref{thm:CYnu1}. 

\begin{thm}
Let $X$ be a projective manifold with $K_X \sim_\Q 0$ and let $\mathcal L$ be a nef line bundle on $X$ with $\nu(X,\mathcal L)=1$. Let $\eta\colon \widetilde X \to X$ be a finite \'etale cover such that the Beauville-Bogomolov decomposition is of the form
$$\widetilde X \simeq T\times \prod X_j,$$ 
where the $X_j$ are even-dimensional Calabi-Yau manifolds or hyperk\"ahler manifolds, and $T$ is an abelian variety. Then 
there exists a line bundle $\mathcal L'$ numerically equivalent to $\mathcal L$ such that $\kappa (X,\mathcal L') \geq 0$. 
\end{thm}

\begin{proof}
Denote $Y =  \prod X_j$, and let $p_1\colon Y \times T \to Y$ and $p_2\colon Y \times T \to T$  be the projections. Since $Y$ is simply connected, by \cite[Exercise III.12.6]{Har77} there exist line bundles $\mathcal M$ and $\mathcal N$ on $Y$ and $T$, respectively, such that
$$\eta^*\mathcal L \simeq p_1^*\mathcal M \otimes p_2^*\mathcal N.$$ 
By restricting this relation to fibres of $p_1$ and $p_2$, we obtain that $\mathcal M$ and $\mathcal N$ are nef. As $T$ is an abelian variety, there exists a semiample line bundle $\mathcal N'$ on $T$ numerically equivalent to $\mathcal N$, and since $\nu(Y,\mathcal M) \leq 1$, we have $\kappa(Y,\mathcal M) \geq 0$ by Theorem \ref{thm:CYnu1}. Therefore, there exists a divisor $D$ on $\widetilde X$ with $\kappa(\widetilde X,D)\geq0$ such that $\eta^*\mathcal L\equiv\OO_{\widetilde X}(D)$, hence $\mathcal L^{\otimes\deg\eta}\equiv\OO_X(\eta_*D)$. This finishes the proof.
\hfill $\Box$
\end{proof}

Following \cite[\S8]{GKP16}, we say that a canonical variety $X$ of dimension $n$ and with $K_X\sim0$ is a Calabi-Yau variety if for every quasi-\'etale cover $\widetilde X\to X$ we have $H^0(\widetilde X,\Omega_{\widetilde X}^{[q]}) = 0$ for $q=1,\dots,n -1$; and that it is a singular irreducible symplectic variety if there is a non-degenerate reflexive holomorphic $2$-form $\omega$ on $X$ such that for every quasi-\'etale cover $f\colon \widetilde X\to X$, every reflexive holomorphic form on $\widetilde X$ is of the form $c f^*\omega^{[p]}$ with a constant $c$. These, together with abelian varieties, are conjecturally the building blocks of singular varieties with trivial canonical class. Then we have:

\begin{thm} \label{sing} 
Let $X$ be a normal projective klt variety such that $K_X \sim_\Q 0$. Suppose that there exists a quasi-\'etale cover $\eta\colon \widetilde X \to X$, such that $\widetilde X$ is either a Calabi-Yau variety of even dimension or a singular irreducible symplectic variety. Let $\mathcal L$ be a nef line bundle on $X$ with $\nu(X,\mathcal L)=1$. Then $\kappa(X,\mathcal L)\geq0$.
\end{thm} 
 
\begin{proof} 
By \cite[Proposition 6.9]{GKP16} we have
$$ H^0(\widetilde X,\Omega^{[q]}_{\widetilde X}) \simeq H^q({\widetilde X},\OO_{\widetilde X}). $$
Therefore $\chi(\widetilde X,\OO_{\widetilde X}) > 0$, and we conclude as in Theorem \ref{thm:CYnu1} that $\kappa(X,\mathcal L)=\kappa(\widetilde X,\eta^*\mathcal L)\geq0$.
\hfill $\Box$
\end{proof} 
 
We also note:
 
\begin{cor} \label{cor:nef1}  
Let $X$ be a projective klt variety of dimension $4$ such that $K_X\sim_\Q0$, and let $\mathcal L$ be a nef line bundle on $X$ with $\nu(X,\mathcal L)=1$. Let $\eta\colon \widetilde X \to X$ be the canonical cover and assume that $h^1(\widetilde X,\OO_{\widetilde X}) = 0$. Then $\kappa (X,\mathcal L) \geq 0$. If $\mathcal L$ is hermitian semipositive, then it is semiample.
\end{cor} 

\begin{proof} 
By \cite[Proposition 6.9 and Corollary 6.11]{GKP16}, we have
$$ h^3(\widetilde X,\mathcal O_{\widetilde X}) = h^1(\widetilde X, \mathcal O_{\widetilde X}) = 0. $$
This implies $\kappa(X,\mathcal L)\geq0$ as in the proof of Theorem \ref{sing}. The last claim follows as in the proof of Corollary \ref{cor:nef}(ii).
\hfill $\Box$
\end{proof} 

\begin{rem}{\rm
Several results of this section can be stated appropriately for klt pairs $(X,\Delta)$ such that $K_X+\Delta\sim_\Q0$; they are often called \emph{varieties of Calabi-Yau type}. The proofs of those generalisations are straightforward adaptations of the proofs presented above.}
\end{rem}

\vskip 5mm

\noindent {\bf Acknowledgements.} Lazi\'c was supported by the DFG-Emmy-Noether-Nachwuchsgruppe ``Gute Strukturen in der h\"oherdimensionalen birationalen Geometrie". Peternell was supported by the DFG grant ``Zur Positivit\"at in der komplexen Geometrie". We would like to thank D.\ Huybrechts, B.\ Taji and L.\ Tasin for useful conversations related to this work. Last but not least, we would like to thank the referee whose comments were extremely valuable.

\providecommand{\bysame}{\leavevmode\hbox to3em{\hrulefill}\thinspace}
%
%

\bibliographystyle{amsalpha}
\bibliographymark{References}

\end{document}